\documentclass{amsart}

\usepackage[T1]{fontenc}
\usepackage[hidelinks]{hyperref}
\usepackage{cleveref}
\usepackage{amsthm}
\usepackage{quiver}

\usepackage{import}

\usepackage{graphicx}

\usepackage{tikzit}

\tikzstyle{dotdotdot}=[fill=white, draw=black, shape=rectangle]
\tikzstyle{new style 0}=[fill=black, draw=none, shape=circle]

\tikzstyle{new edge style 0}=[-, densely dashed]
\tikzstyle{new edge style 2}=[-, draw=blue]
\tikzstyle{new edge style 3}=[->]

\usepackage{caption}
\usepackage{xparse}

\numberwithin{figure}{section}

\newtheorem{theorem}{Theorem}[section]
\newtheorem*{theorem*}{Theorem}

\newtheorem{lemma}[theorem]{Lemma}
\newtheorem{prop}[theorem]{Proposition}
\newtheorem*{prop*}{Proposition}

\newtheorem{coro}[theorem]{Corollary}
\theoremstyle{remark}
\newtheorem{remark}[theorem]{Remark}
\newtheorem*{question*}{Question}
\theoremstyle{definition}
\newtheorem{definition}[theorem]{Definition}

\Crefname{theorem}{Theorem}{Theorems}
\Crefname{prop}{Proposition}{Propositions}
\Crefname{question}{Question}{Questions}
\Crefname{coro}{Corollary}{Corollaries}
\Crefname{lemma}{Lemma}{Lemmas}
\Crefname{definition}{Definition}{Definitions}
\Crefname{remark}{Remark}{Remarks}

\newcommand{\SO}{\mathrm{SO}(3)}

\newcommand{\scc}{simple closed curve}
\newcommand{\sccs}{simple closed curves}
\newcommand{\und}[1]{\underline{#1}}

\newcommand{\lra}{\longrightarrow}
\newcommand{\ra}{\rightarrow}
\newcommand{\xlra}{\xrightarrow}

\newcommand{\Lra}{\Longrightarrow}

\newcommand{\Nu}{\mathcal{V}}
\newcommand{\ad}{\mathrm{ad}\,}

\newcommand{\PMod}[1]{\mathrm{Mod}(#1)}
\newcommand{\tPMod}[1]{\mathrm{M\tilde{o}d}(#1)}
\newcommand{\PModl}[2]{\mathrm{Mod}^{#1}(#2)}
\newcommand{\tPModl}[2]{\mathrm{M\tilde{o}d}^{#1}(#2)}
\newcommand{\coh}[2]{\mathrm{H}^1(#1,#2)}

\newcommand{\End}[1]{\mathrm{End}(#1)}
\newcommand{\GL}[1]{\mathrm{GL}(#1)}
\newcommand{\PU}[2]{\mathrm{PU}(#1,#2)}

\newcommand{\TU}[1]{\mathrm{TU}(#1)}
\newcommand{\Ug}[1]{\mathrm{U}(#1)}
\newcommand{\Id}[1]{\mathrm{I}_{#1}}

\newcommand{\Aut}[1]{\mathrm{Aut}(#1)}

\newcommand{\Hom}[2]{\mathrm{Hom}(#1,#2)}

\newcommand{\Z}{\mathbb{Z}}

\newcommand{\C}{\mathbb{C}}
\newcommand{\Q}{\mathbb{Q}}
\newcommand{\R}{\mathbb{R}}

\newcommand{\id}{\mathrm{id}}

\title{Rigidity of Fibonacci representations of mapping class groups}
\author{Pierre Godfard}

\begin{document}

\begin{abstract} 
    We prove that level $5$ Witten-Reshetikhin-Turaev \texorpdfstring{$\SO$}{SO(3)} quantum representations,
    also known as the Fibonacci representations,
    of mapping class groups are locally rigid.
    More generally, for any prime level $\ell$, we prove that the level $\ell$ \texorpdfstring{$\SO$}{SO(3)} quantum representations
    are locally rigid on
    all surfaces of genus $g\geq 3$ if and only if they are locally rigid on surfaces of genus $3$ with at most $3$ boundary components.
    This reduces local rigidity in prime level $\ell$ to a finite number of cases.
\end{abstract}

\maketitle

\section{Introduction}

It is expected that TQFT representations are locally rigid,
either because of Kazhdan's property (T) (yet to be proved for mapping class groups),
or because of interpretations as complex variations of Hodge structures (yet to be constructed, see \ref{CVHS} below).

In this paper, we show that local rigidity at a prime level $\ell$ reduces to a finite number of cases.
This enables us to prove local rigidity in the Fibonacci case (level $5$).

\subsection{The results}\label{subsectionintroresults}

The quantum representations studied here are projective representations of the mapping class groups of surfaces,
parametrized by a Lie group $G$, a positive integer $\ell$ called level,
and a set of irreducible representations $\Lambda$ of $G$, that depends on $\ell$.
We will restrict ourselves to the case where $G$ is $\SO$ and the level is prime, with a complete result only for $\ell=5$.
In quantum topology, these representations arise from the Witten-Reshetikhin-Turaev TQFT \cite{reshetikhinInvariants3manifoldsLink1991}.
They are known to be related to spaces of conformal blocks.

The level has another meaning in Conformal Field Theory, but here the level will be the order of the root of unity considered.
More precisely, these representations are defined over the cyclotomic field $\Q(\zeta_\ell)$
(and even over its ring of integers \cite{gilmerIntegralLatticesTQFT2007}).

In this context, for each compact surface $S_g^n$ of genus $g$ with $n$ boundary components,
odd integer $\ell$ and $n$-tuple of colors $\underline{\lambda}\in \Lambda^n$,
there is a representation of the mapping class group $\PMod{S_g^n}$:
$$\rho_{g,n}(\underline{\lambda}):\PMod{S_g^n}\lra \mathrm{PGL_d}(\Q(\zeta_\ell))$$
where $d$ depends on $g$, $n$ and $\underline{\lambda}$.

In this paper, we study the local rigidity of these representations. More\linebreak precisely, with
the notation $\mathcal{X}(\PMod{S_g^n},\mathrm{PGL_d})$ for the character variety\linebreak
$\Hom{\PMod{S_g^n}}{\mathrm{PGL_d}} // \mathrm{PGL_d}$, the question is the following.

\begin{question*}
Is $\rho_{g,n}(\underline{\lambda})$ locally rigid? ie.
do we have $$\coh{\PMod{S_g^n}}{\ad\rho_{g,n}(\underline{\lambda})}=0?$$
Equivalently, is $[\rho_{g,n}(\underline{\lambda})]$ an isolated smooth point of $\mathcal{X}(\PMod{S_g^n},\mathrm{PGL_d})$?
\end{question*}
This question was asked, for example,
in the survey by L. Funar \cite[Question 2.3]{funarMappingClassGroup2021c}.
Here, $\ad \rho_{g,n}(\underline{\lambda})$ is the adjoint representation of $\PMod{S_g^n}$
on the space $\mathfrak{sl}_{\mathrm{d}}(\C)$ of matrices with trace $0$.

These representations factor through $\PModl{\ell}{S_g^n}$, the quotient of $\PMod{S_g^n}$ by the $\ell$-th powers of Dehn twists.
Our main result is the following.

\begin{theorem*}[\textbf{\ref{rigidityinlowlevel}}]
    For $g\geq 3$ and $n\geq 0$, the representations of $\PMod{S_g^n}$ coming from the $\SO$ TQFT in level $5$ are locally rigid within
    $\mathrm{PGL_d}(\C)$.
    Moreover, if we consider these representations as representations of the quotients $\PModl{5}{S_g^n}$ of the mapping class groups,
    then they are all locally rigid within $\mathrm{PGL_d}(\C)$, for any genus $g\geq 0$
    and number of boundary components $n\geq 0$.
\end{theorem*}

Note that in this context, local rigidity never depends on the target group of the representation (see \Cref{reductiontounitary})
and, when $g\geq 3$, does not depend on whether the source group is $\PMod{S_g^n}$ or $\PModl{\ell}{S_g^n}$ (see \Cref{removingthel}).

The proof relies on a general induction on $g$ and $n$
and reduces the proof to a finite number of cases, where rigidity has then to be proved directly.
This can be done either by computer, or by geometric arguments.
For example, in the case $S_0^5$ of genus $0$ with $5$ boundary components, the representation is the monodromy of the Hirzebruch
surface, and local rigidity follows from Weil's rigidity.

More generally, for any prime level $\ell\geq 5$, there is an induction process that reduces the study of
local rigidity to the case of genus $3$ with at most $3$
boundary components:

\begin{theorem*}[\textbf{\ref{induction}}]
    Let $\ell\geq 5$ be a prime number.

    Assume that for $n\in\{0,1,2,3\}$ and every coloring $\und{\lambda}$ of $\partial S_3^n$,
    the corresponding $\SO$ representation of $\PMod{S_3^n}$ is locally rigid within $\mathrm{PGL_d}(\C)$.

    Then for all $g\geq 3$, $n\geq 0$ and coloring $\und{\lambda}$ of $\partial S_g^n$,
    the associated $\SO$ representation of $\PMod{S_g^n}$ is locally rigid within $\mathrm{PGL_d}(\C)$.
\end{theorem*}

\begin{remark}\label{remarkanomaly}
    In the rest of the paper, we need to use a linearized version of the representations to state the results.
    This amounts to take in account the "projective anomaly" of quantum representations.
    This approach is equivalent, see \Cref{appendixrigidity} for details.
\end{remark}

\begin{remark}
    In prime level $\ell>5$, the representation of $\PModl{\ell}{S_0^4}$ coming from the $\SO$ TQFT in level $\ell$
    is not locally rigid for $\underline{\lambda}=(2,2,2,2)$. This can be proved by a dimension count using the method of \Cref{S04}.
    Thus \Cref{rigidityinlowlevel} fails at genus $0$ for larger prime levels.
    This does not stop us from conjecturing that the representations are locally rigid in genus $g\geq 3$ for all prime levels.
\end{remark}

\subsection{Existence of complex variation of Hodge structures}\label{CVHS}

For $g,n\geq 0$ with $3g-3+n\geq 0$ and $\ell\geq 5$ odd, the group $\PModl{\ell}{S_g^n}$ is the fundamental group
of a compact Kähler uniformizable orbifold $\overline{\mathcal{M}}_{g,n}(\ell)$ (\cite[1.1]{eyssidieuxOrbifoldKahlerGroups2021}).
It is constructed as a moduli space of $\ell$-twisted stable curves.

As this orbifold is uniformizable and Kähler, it verifies the non-abelian Hodge correspondence.
Thus any rigid reductive representation of its fundamental group supports a complex variation of Hodge structure
\cite[Lemma 4.5]{simpsonHiggsBundlesLocal1992}.

\begin{theorem}
    For $g\geq 0$ and $n\geq 0$ with $3g-3+n\geq 0$, the flat projective bundles over $\overline{\mathcal{M}}_{g,n}(5)$ induced by
    the level $5$ $\SO$ TQFT representations of $\PModl{5}{S_g^n}$ support complex variations of Hodge structure.
\end{theorem}

C. Simpson's motivicity conjecture states that the complex variations of Hodge structure of rigid representations have geometric origin.
In our situation, this would mean that the flat projective bundle on $\overline{\mathcal{M}}_{g,n}(5)$
could be constructed from a sub flat bundle of the cohomology $R^mp_*\C$ of some fibration $p:E_{g,n}\ra \overline{\mathcal{M}}_{g,n}(5)$.

The existence of a complex variation of Hodge structure thus makes a link between quantum representations and geometry.
It would be interesting to study the Hodge decompositions and their compatibilities as a family in $g$ and $n$.

\subsection{Rigidity as unitary representations of the mapping class group}

After the field embedding $i:\Q(\zeta_\ell)\ra\C$ given by $\zeta_\ell=\exp(\pm i\pi\frac{\ell-1}{\ell})$,
the representations are known to be unitary.
Hence \Cref{rigidityinlowlevel} has the following corollary.

\begin{theorem}
    For $g\geq 3$ and $n\geq 0$, the representations of $\PMod{S_g^n}$ coming from the $\SO$ TQFT in level $5$ are locally rigid
    as unitary representations.
\end{theorem}

Local rigidity of unitary representations is implied by Kazhdan's property (T).
It is not known if $\PMod{S_g}$ verifies property (T) for $g\geq 3$.
However, $\PMod{S}$ does not verify property (T) if the genus of $S$ is at most $2$,
as it has a finite index subgroup that surjects onto $\Z$ (see, for example, \cite[2.3]{aramayonaRigidityPhenomenaMapping2016}).

\subsection{Outline of the proofs}

Here we outline the proofs of \Cref{induction} and \Cref{rigidityinlowlevel}. Both proofs rely on \Cref{mainlemma}.
Given a surface $S$ and $2$ simple closed curves $a_1$, $a_2$ on $S$, we denote by $S_{a_i}$ the surface obtained by cutting $S$ along $a_i$.
Under some assumptions, the lemma relates the deformation space of a representation of $\tPModl{\ell}{S}$ to the deformation spaces 
of its restrictions to $\tPModl{\ell}{S_{a_1}}$ and $\tPModl{\ell}{S_{a_2}}$. As a consequence, we will see that if the quantum representations
associated to $S_{a_1}$ and $S_{a_2}$ are locally rigid, then so are those associated to $S$ (see \Cref{bettermainlemma}).
The proof of \Cref{mainlemma} is the content of \Cref{sectionmainlemma}. 

\Cref{mainlemma} enables us to prove \Cref{induction} by induction on the genus $g$ and the number $n$ of boundary components.
The proof of \Cref{induction} is the content of \Cref{sectioninduction}.

The proof of \Cref{rigidityinlowlevel} is in $2$ steps. The first step is to prove rigidity for some small surfaces.
More precisely, local rigidity for $S_0^2$, $S_0^3$, $S_1$ and $S_1^1$
is deduced from the finiteness of the associated mapping class group quotients.
Local rigidity for $S_0^4$ is proved with a direct dimension computation. For $S_0^5$, the proof relies on Weil's rigidity.
The second step is to perform an induction on $g$ and $n$ as in the proof of \Cref{induction}.
However, because the assumptions of \Cref{mainlemma} are not always verified,
we have to adapt the techniques of the lemma to the various cases. The proof of \Cref{rigidityinlowlevel} is the content
of \Cref{sectionrigidityinlowlevel}.

\subsection*{Acknowledgements}

This paper forms part of the PhD thesis of the author.
The author thanks J. Marché for his help in writing this paper.
The author also thanks B. Deroin, R. Detcherry, L. Funar, G. Masbaum and R. Santharoubane for helpful discussions.
Finally, the author would like to thank the Reviewer for taking the necessary time and effort to review the manuscript.

\section{Modular functors}\label{sectionmodularfunctor}

\subsection{Definition}

Even though our proof applies essentially only to the $\SO$ representations at prime levels,
we will write it in the context of modular functors to make clear which properties of the TQFT
are used. In particular, with the exception of \Cref{appendixirreducibility}, we will not refer explicitly to the construction of the TQFT.
For the general theory of modular functor, see Turaev's book \cite[chapter 5]{turaevQuantumInvariantsKnots2016}.
We will use a stripped down version of modular functors, similar to the one used by B. Deroin and J. Marché in
\cite[section 4]{deroinToledoInvariantsTopological2022}.

\begin{definition}\label{definitioncolourset}
    A set of colors is a finite set $\Lambda$ with a preferred element $0\in \Lambda$.
\end{definition}

\begin{remark}
    In the usual definition of a modular functor, there is the extra datum of an involution $\lambda\mapsto\lambda^*$
    on $\Lambda$. Here, for simplicity, the involution is assumed to be the identity, as it is trivial in our examples.
\end{remark}

\begin{definition}\label{definitionsurfaces}
    Let $g,n\geq 0$. We define $S_g^n$ to be the compact surface of genus $g$ with $n$ boundary components,
    and $S_{g,n}$ to be the surface of genus $g$ with $n$ punctures and no boundary.
\end{definition}

\begin{definition}\label{definitionlagrangian}
    Let $S$ be a compact surface, which can have non-empty boundary.
    Let $\widehat{S}$ be the closed surface obtained from $S$ by capping off each boundary component $S^1$ with $D^2$.
    A Lagrangian $L$ on $S$ is a subspace of $H_1(\widehat{S};\Q)$ of half dimension on which the intersection form vanishes.
    A split Lagrangian on $S$ is a Lagrangian $L$ that is a direct sum of Lagrangians on the connected components of $S$.
\end{definition}

\begin{definition}\label{definitionmaslovindex}
    Let $S$ be a compact oriented surface and $L_0,L_1$ and $L_2$ be $3$ split Lagrangians on $S$.
    Their Maslov index, denoted $\mu(L_0,L_1,L_2)$,
    is the signature of the quadratic form $q$ defined as follows. Let $K$ be the kernel of the sum map:
    $$L_0\oplus L_1\oplus L_2\lra H_1(\widehat{S};\Q).$$
    Then we define $q$ as:
    $$q: K \lra \Q,\;(u_0,u_1,u_2) \longmapsto u_0 \cdot u_1$$
    where $\cdot$ is the intersection form.
\end{definition}

We can now define the source category of modular functors.

\begin{definition}\label{definitioncolouredcategory}
    Let $\Lambda$ be a set of colors. The category of surfaces colored with $\Lambda$ is such that:
    \begin{description}
        \item[(1)] its objects are compact oriented surfaces $S$ together with a Lagrangian $L$ on $S$,
        an identification $\varphi_B: B\simeq S^1$ and a color $\lambda_B\in \Lambda$ for every component $B$ of $\partial S$ ;
        \item[(2)] its morphisms from $\Sigma_1=(S_1,L_1,\varphi^1,\underline{\lambda}^1)$
        to $\Sigma_2=(S_2,L_2,\varphi^2,\underline{\lambda}^2)$
        are pairs $(f,n)$ with $f:S_1\lra S_2$ an homeomorphism preserving orientation such for every component $B_1\subset\partial S_1$
        and its image $f(B_1)=B_2\subset \partial S_2$, we have $\lambda_{B_1}=\lambda_{B_2}$ and $\varphi^2_{B_2}\circ f=\varphi^1_{B_1}$.
        The second element $n$ of the pair is an integer in $\Z$.
        \item[(3)] the composition of $(f_1,n_1):\Sigma_0\lra\Sigma_1$ and $(f_2,n_2):\Sigma_1\lra\Sigma_2$ is given by:
        $$(f_2\circ f_1,n_1+n_2-\mu(f_1(L_0),L_1,f_2^{-1}(L_2)))$$
        where $L_0, L_1$ and $L_2$ are the respective Lagrangians of $\Sigma_0$, $\Sigma_1$ and $\Sigma_2$.
    \end{description}
    This category has a natural monoidal structure induced by the disjoint union $\sqcup$.
\end{definition}

In the rest of the paper, $\Sigma=(S,L,\varphi,\underline{\lambda})$ will be abbreviated $(S,\underline{\lambda})$, or even
$(S,\lambda_1,\lambda_2,\dotsc)$ where $\lambda_1,\lambda_2,\dotsc$ are the colors relevant to the argument and the other colors are omitted.

\begin{definition}\label{definitioncentralextension}
    Let $S$ be a surface with boundary. We note by $\PMod{S}$ its mapping class group, ie. the group of connected components
    of the group of orientation preserving homeomorphisms of $S$ fixing the boundary $\partial S$ pointwise.

    Let $L$ be a Lagrangian for $S$, and $\Aut{S,L}$ the group of pairs $(f,n)$ with $f\in \PMod{S}$ and $n\in\Z$
    with composition as in \textbf{(3)} of \Cref{definitioncolouredcategory}.

    The isomorphism class of $\Aut{S,L}$ does not depend on $L$. Hence we will use the notation $\tPMod{S}$ for it.
    The group $\tPMod{S}$ is a central extension of $\PMod{S}$ by $\Z$:
    $$1\lra \Z\lra \tPMod{S}\lra \PMod{S}\lra 1.$$
\end{definition}

Let $S$ be a surface and $\partial_+S\sqcup\partial_-S\subset\partial S$ be two components of its boundary.
Let $\varphi_{\partial_{\pm}S}:\partial_{\pm}S\simeq S^1$ be identifications of these components with $S^1$.

Let $S_{\pm}$ be the surface obtained from $S$ by gluing $\partial_+S$ to $\partial_-S$
along $\varphi_{\partial_-S}^{-1}\circ\varphi_{\partial_+S}$.
Then $S_{\pm}$ is called the gluing of $S$ along $\partial_{\pm}S$.

One can check that if $S$ has a Lagrangian $L$, one defines a Lagrangian in $S_{\pm}$ as follows.
There exists a $3$-manifold $M$ with boundary $\widehat{S}$ such that $L$
is the kernel of $H_1(\widehat{S};\Q)\lra H_1(M;\Q)$. Let $M_{\pm}$ be the $3$ manifold obtained by gluing together
the discs bounding $\partial_+S$ and $\partial_-S$ on the boundary of $M$.
Then $M_{\pm}$ bounds $\widehat{S_{\pm}}$ and the desired Lagrangian is the kernel of $H_1(\widehat{S_{\pm}};\Q)\lra H_1(M_{\pm};\Q)$.

We now introduce the notion of modular functor.

\begin{definition}[Modular Functor]\label{definitionmodularfunctor}
    Let $\Lambda$ be a set of colors and $\mathrm{C}$ be the associated category of colored surfaces
    as defined in \Cref{definitioncolouredcategory}.
    Then a modular functor is the data of a monoidal functor:
    $$\Nu:\mathrm{C}\lra \C-\text{vector spaces}$$
    where the monoidal structure on $\C$-vector spaces is understood to be the tensor product.
    This data is augmented by the following isomorphisms.
    \begin{description}
        \item[(G)] For any surface with Lagrangian $(S,L)$ and pair of boundary components $\partial_{\pm}S$,
        let $(S_{\pm},L_{\pm})$ be the gluing of $S$ along $\partial_{\pm}S$. For any coloring $\underline{\lambda}$
        of the components of $\partial S_{\pm}$, an isomorphism as below is given:
        $$\Nu(S_{\pm},L_{\pm},\underline{\lambda})\simeq \bigoplus_{\mu\in\Lambda}\Nu(S,L,\mu,\mu,\underline{\lambda}).$$
    \end{description}
    The isomorphisms of \textbf{(G)} are assumed to be functorial and compatible with disjoint unions.
    This rule, also sometimes called fusion or factorization rule, is the most important property of modular functors.
    The functor is also assumed to verify two more axioms:
    \begin{description}
        \item[(1)] $\dim\Nu(S_0^1,\lambda)=1$ if $\lambda=0$ and $0$ otherwise;
        \item[(2)] $\dim\Nu(S_0^2,\lambda,\mu)=1$ if $\lambda=\mu$ and $0$ otherwise.
    \end{description}
\end{definition}

\begin{remark}\label{remarklevel}
    Let $S_{\pm}$ be a colored surface constructed as a gluing of $S$ along $\partial_{\pm}S$.
    Let $\gamma$ denote the simple closed curve that is the image of $\partial_{\pm}S$ in $S_{\pm}$.
    Then the Dehn twist $T_\gamma$ acts block-diagonally on the decomposition \textbf{(G)}.

    Moreover, one can easily see that it acts on the block $\Nu(S,L,\mu,\mu,\underline{\lambda})$ by a scalar $r_\mu$,
    that depends only on $\mu$, and not on the surface $S$.
    Indeed, $r_\mu$ is given by the action of the unique Dehn twist of $S_0^2$ on $\Nu(S_0^2,\mu,\mu)$, which is $1$-dimensional.

    From the fact that on any finite-dimensional representation of $\tPMod{S}$ with $S$ of genus at least $3$,
    the Dehn twists act with quasi-unipotent matrices \cite[2.5]{aramayonaRigidityPhenomenaMapping2016},
    one deduces that for all $\mu$, $r_\mu$ is a root of unity.

    Similarly, one can prove that $(\id_\Sigma,1)$ acts by a scalar $\kappa$ on $\Nu(\Sigma)$ and that $\kappa$
    is independent of $\Sigma$.
\end{remark}

\begin{definition}
    Let $\Nu$ be a modular functor.
    A level for $\Nu$ is an integer $\ell\geq 1$ such that $\forall\lambda\in\Lambda$,
    $r_\lambda^\ell=1$ and $\kappa^{4\ell}=1$.
    For $\ell>1$ an integer and $S$ a surface, the group $\tPModl{\ell}{S}$ is defined as the quotient of $\tPMod{S}$
    by the subgroup generated by $(\id,1)^{4\ell}$ and the $\ell$-th powers of the Dehn twists $(T_\gamma,0)^\ell$
    for every $\gamma$ {\scc} in the Lagrangian $L$ such that $\tPMod{S}=\Aut{S,L}$.
\end{definition}

We can now define the representations studied in this article:

\begin{definition}[Quantum representations]
    Let $\Nu$ be a modular functor and $\ell$ be a level for $\Nu$.
    Then, for any surface $S$ and coloring $\underline{\lambda}$ of its boundary components, the functor yields a representation:
    $$\rho_{g,n}(\underline{\lambda}):\tPMod{S}\lra \GL{\Nu(S,\underline{\lambda})}$$
    which factors as a representation:
    $$\rho_{g,n}^\ell(\underline{\lambda}):\tPModl{\ell}{S}\lra \GL{\Nu(S,\underline{\lambda})}.$$

    These latter representations will be called the representations associated to the modular functor $\Nu$.

    We shall say that $\Nu$ is rigid on $(S,\underline{\lambda})$ if $\rho_{g,n}^\ell(\underline{\lambda})$ is cohomologically rigid,
    ie. if $\coh{\tPModl{\ell}{S}}{\ad \rho_{g,n}^\ell(\underline{\lambda})}=0$.
\end{definition}

We now introduce extra properties that we will need in our proof of rigidity:
\begin{description}
    \item[(I)] For every $\lambda\in\Lambda$, $\Nu(S_1^1,\lambda)\neq 0$;
    \item[(II)] For every $\lambda,\mu,\nu\in\Lambda$, $\Nu(S_0^3,\lambda,\mu,\nu)$ has dimension $0$ or $1$;
    \item[(III)] For every $\lambda,\mu\in\Lambda$, if $\lambda\neq\mu$, then $r_\lambda\neq r_\mu$.
\end{description}
\textbf{(I)} can be equivalently rephrased:
\begin{description}
    \item[(I)] For every $\lambda\in\Lambda$, there exists $\mu\in\Lambda$ such that $\Nu(S_0^3,\lambda,\mu,\mu)\neq 0$.
\end{description}

To the author's knowledge, these properties essentially restrict to the case of the $\SO$ modular functors at prime levels.

\subsection{Bases}\label{modularfunctorbases}

We will now explain how to construct bases of the $\Nu(S,\underline{\lambda})$ for a modular functor $\Nu$ satisfying \textbf{(II)}.

\begin{definition}
    Let $\Nu$ be a modular functor satisfying \textbf{(II)}.
    For $\lambda,\mu,\nu\in\Lambda$, we say that $(\lambda,\mu,\nu)$ is admissible if $\Nu(S_0^3,\lambda,\mu,\nu)\neq 0$.
\end{definition}

Let $(S,\underline{\lambda})$ be a colored surface and $\{e_i\}$ be a set of disjoint simple closed curves on $S$
that induces a pair of pants decomposition, ie. such that cutting along the curves yields a disjoint union of surfaces,
each homeomorphic to $S_0^3$. Let us denote $S_{\mathrm{cut}}$ this disjoint union.

\begin{figure}
    \def\svgwidth{0.8\linewidth}
\begingroup%
  \makeatletter%
  \providecommand\color[2][]{%
    \errmessage{(Inkscape) Color is used for the text in Inkscape, but the package 'color.sty' is not loaded}%
    \renewcommand\color[2][]{}%
  }%
  \providecommand\transparent[1]{%
    \errmessage{(Inkscape) Transparency is used (non-zero) for the text in Inkscape, but the package 'transparent.sty' is not loaded}%
    \renewcommand\transparent[1]{}%
  }%
  \providecommand\rotatebox[2]{#2}%
  \newcommand*\fsize{\dimexpr\f@size pt\relax}%
  \newcommand*\lineheight[1]{\fontsize{\fsize}{#1\fsize}\selectfont}%
  \ifx\svgwidth\undefined%
    \setlength{\unitlength}{253.71149258bp}%
    \ifx\svgscale\undefined%
      \relax%
    \else%
      \setlength{\unitlength}{\unitlength * \real{\svgscale}}%
    \fi%
  \else%
    \setlength{\unitlength}{\svgwidth}%
  \fi%
  \global\let\svgwidth\undefined%
  \global\let\svgscale\undefined%
  \makeatother%
  \begin{picture}(1,0.37170022)%
    \lineheight{1}%
    \setlength\tabcolsep{0pt}%
    \put(0,0){\includegraphics[width=\unitlength,page=1]{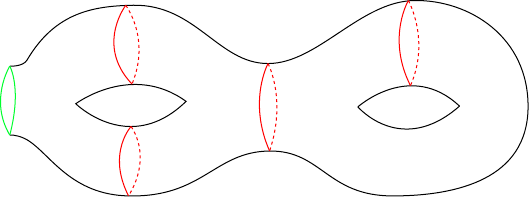}}%
    \put(0.00581519,0.26896511){\color[rgb]{0,1,0.18431373}\makebox(0,0)[lt]{\lineheight{1.25}\smash{\begin{tabular}[t]{l}$\lambda$\end{tabular}}}}%
    \put(0,0){\includegraphics[width=\unitlength,page=2]{graph_example.pdf}}%
  \end{picture}%
\endgroup%

    \caption{Pair of pants decomposition of a surface and associated graph, with leg in green.}
    \label[figure]{graph_example}
\end{figure}

Let $G=(V,H,v:H\ra V,\iota:H\ra H)$ be the trivalent graph defined by vertices and half-edges such that:
\begin{itemize}
    \item the vertices $V$ are the components of $S_{\mathrm{cut}}$;
    \item the half-edges $H$ are the components of $\partial S_{\mathrm{cut}}$;
    \item the attaching map $v:H\ra V$ sends a half edge $h$ to the component of $S_{\mathrm{cut}}$ on which it sits;
    \item $\iota$ is an involution of $H$;
    \item The set $E$ of edges of $G$ is the set of $2$-cycles of $\iota$. It identifies with the chosen set of curves $\{e_i\}$ on $S$;
    \item The set $L$ of legs of $G$ is the set of fixed points of $\iota$. It identifies with the set of components of $\partial S$.
\end{itemize}
An example is given on \Cref{graph_example}.

By a coloring of $G$, we mean a map $c:E\lra \Lambda$. We extend a coloring to $L$ by sending a leg to its associated color in $S$.
Alternatively, we can see $c$ as a map $c:H\lra \Lambda$ such that $c\circ\iota=c$ and $c_{\mid L}$ coincides with the coloring of the components of $\partial S$.

Such a coloring $c$ is said to be admissible if for any vertex $v$ and adjacent half-edges $h_1$, $h_2$ and $h_3$,
the triplet $(c(h_1),c(h_2),c(h_3))$ is admissible. We denote by $C(G)$ the set of admissible colorings of $G$.

By using the gluing axiom \textbf{(G)} repeatedly, one gets a decomposition:
$$\Nu(S,\underline{\lambda})=\bigoplus_{c:E\rightarrow \Lambda}\bigotimes_{v\in V}\Nu(S_0^3,c(h_1),c(h_2),c(h_3))$$
where the $h_i$ are the half-edges adjacent to $v$.

We can remove the non-admissible colorings to obtain a decomposition of $\Nu(S,\underline{\lambda})$ in vector spaces of dimension $1$:
$$\Nu(S,\underline{\lambda})=\bigoplus_{c\in C(G)}\bigotimes_{v\in V}\Nu(S_0^3,c(h_1),c(h_2),c(h_3)).$$

In particular, $\dim\Nu(S,\underline{\lambda})=\left\vert C(G)\right\vert$.

\subsection{\texorpdfstring{The $\SO$ modular functors}{The SO(3) modular functors}}

Let $\ell\geq 5$ be an odd integer. Let $\Lambda=\{0,2,\dotsc,\ell-3\}$.

Then there exists a modular functor $\Nu_\ell$ on the surfaces colored with $\Lambda$,
called the $\SO$ modular functor of level $\ell$.

We refer to the work of C. Blanchet, N. Habegger, G. Masbaum, and P. Vogel \cite{blanchetTopologicalQuantumField1995} for a construction of the $\SO$ modular functors.
Only the following properties of these functors will be used in this paper.

\begin{prop}\label{SOmodularfunctor}
    Let $\ell\geq 5$ be an odd number. Then the modular functor $\Nu_\ell$ satisfies properties \textbf{(I)} and \textbf{(II)}.

    Moreover, if $\ell$ is a prime number, $\Nu_\ell$ satisfies property \textbf{(III)}.
\end{prop}

\begin{proof}
    For properties \textbf{(I)} and \textbf{(II)}, see \cite{blanchetTopologicalQuantumField1995}.
    Property \textbf{(III)} is easily deduced from the fact that $r_\lambda=\zeta_\ell^{\lambda(\lambda+2)}$.
    See \cite[section 3]{blanchetThreemanifoldInvariantsDerived1992} or \cite[Lemma 2.5]{marcheIntroductionQuantumRepresentations2021}
    for the computation of $r_\lambda$.
\end{proof}

\begin{prop}\label{colourconditions}
    Let $a,b,c\in\Lambda=\{0,2,\dotsc,\ell-3\}$. Then $(a,b,c)$ is admissible for $\Nu_\ell$ if and only if:
    \begin{itemize}
        \item $a,b,c$ verify triangular inequalities, ie. $|a-b|\leq c \leq a+b$;
        \item $a+b+c<2\ell-2$.
    \end{itemize}
\end{prop}

The following proposition follows from a result of \cite{koberdaIrreducibilityQuantumRepresentations2018},
where they apply the method of \cite{robertsIrreducibilityQuantumRepresentations2001} to a cleverly chosen special case of a proposition of
\cite{blanchetTopologicalQuantumField1995}. See \Cref{appendixirreducibility} for details. 

\begin{prop*}[\textbf{\ref{irreducibility}}]
    Let $\ell\geq 5$ be a prime number. Then for any colored surface $(S,\underline{\lambda})$,
    the representation of $\tPMod{S}$ on $\Nu_\ell(S,\underline{\lambda})$ is irreducible.
\end{prop*}

\section{Main results}

\subsection{General induction results}

We say that a modular functor $\Nu$ is irreducible on a surface $S$ if for any coloring $\underline{\lambda}$ of the components of $\partial S$,
the representation of $\tPMod{S}$ on $\Nu(S,\underline{\lambda})$ is irreducible.
More generally, if $\partial S=\partial_cS\sqcup\partial_uS$ and $\underline{\mu}$ is a coloring of the components of $\partial_cS$,
we say that $\Nu$ is irreducible on $(S,\underline{\mu})$ if for any coloring $\underline{\lambda}$ of $\partial_uS$,
the representation of $\tPMod{S}$ on $\Nu(S,\underline{\mu},\underline{\lambda})$ is irreducible.

By extension, we say that $\Nu$ is irreducible in genus $g\geq k$ if for any surface $S$ of genus at least $k$,
$\Nu$ is irreducible on $S$.

The following theorem reduces local rigidity in genus $g\geq 3$ to local rigidity in genus $3$ with at most $3$ boundary components.

\begin{theorem}\label{induction}
    Let $\Nu$ be a modular functor satisfying \textbf{(I)}, \textbf{(II)} and \textbf{(III)} that is irreducible in genus $g\geq 1$.
    Let $\ell>1$ be a level for $\Nu$.

    Assume that for $n\in\{0,1,2,3\}$ and every coloring $\und{\lambda}$ of $\partial S_3^n$, the representation of $\tPMod{S_3^n}$
    on $\Nu(S_3^n,\und{\lambda})$ is locally rigid within $\mathrm{GL}_d(\C)$.

    Then for all $g\geq 3$, $n\geq 0$ and coloring $\und{\lambda}$ of $\partial S_g^n$, the representation of $\tPMod{S_g^n}$
    on $\Nu(S_g^n,\und{\lambda})$ is locally rigid within $\mathrm{GL}_d(\C)$.
\end{theorem}

Note that in this context, local rigidity never depends on the target group of the representation (see \Cref{reductiontounitary})
and, when $g\geq 3$, does not depend on whether the source group is $\PMod{S_g^n}$ or $\PModl{\ell}{S_g^n}$ (see \Cref{removingthel}).

\begin{remark}
    For a fixed prime level $\ell$, this theorem reduces the local rigidity of the level $\ell$
    $\SO$ representations in genus $g\geq 3$ to a finite number of cases.
    Notice that as the mapping class groups have known finite presentations,
    the vanishing of $\coh{\tPMod{S}}{\ad\rho_{g,n}(\underline{\lambda})}$
    on a fixed colored surface $(S,\underline{\lambda})$
    can be translated into the computation of the rank of a finite linear system.
    Such a computation can be fed into a computer.

    Thus, given a prime level $\ell$, the theorem reduces the question of local rigidity
    to a problem solvable by computation (that is, if the answer happens to be yes,
    otherwise the computation will just show that some of these representations are not rigid).
\end{remark}

\subsection{\texorpdfstring{The case of the $\SO$ TQFT in level 5}{The case of the SO(3) TQFT in level 5}}

\begin{theorem}\label{rigidityinlowlevel}
    For $g\geq 0$ and $n\geq 0$, the representations of $\tPModl{5}{S_g^n}$ coming from the $\SO$ TQFT of level $5$ are locally rigid
    within $\mathrm{GL}_d(\C)$.
    They are also locally rigid as representations of $\tPMod{S_g^n}$ when $g\geq 3$.
\end{theorem}

\begin{remark}
    In level $5$ the $\SO$ TQFT has only $2$ colors: $0$ and $2$.
    \Cref{removezeros} shows that if a boundary component is colored with $0$,
    then it can be capped off without changing the deformation space.
    Thus, when proving \Cref{rigidityinlowlevel}, we can assume all the boundary components are colored with $2$.
\end{remark}

The method of the proof is in $2$ steps:
\begin{itemize}
    \item Prove local rigidity directly for $S_0^2$, $S_0^3$, $S_0^4$, $S_0^5$, $S_1$ and $S_1^1$,
    using algebraic or geometric arguments;
    \item Use \Cref{induction} or similar arguments to perform an induction on the genus and the number of boundary components.
\end{itemize}

\section{Deformations}\label{appendixrigidity}

In this section, we relate the different versions of local rigidity for the quantum representations.

Let us first relate local rigidity as linear representation and projective representation.
If $\rho:G\lra \mathrm{GL_d}(\C)$ is a linear representation, we denote by $\ad \rho$
the adjoint representation of $G$ on the space $\mathfrak{gl}_{\mathrm{d}}(\C)$.
If $\rho:G\lra \mathrm{PGL_d}(\C)$ is a projective representation, we denote by $\ad \rho$
the adjoint representation of $G$ on the space $\mathfrak{sl}_{\mathrm{d}}(\C)$ of matrices with trace $0$.
This is because the tangent space to $\mathrm{PGL_d}(\C)$ at $I_d$ is isomorphic to $\mathfrak{sl}_{\mathrm{d}}(\C)$.

\begin{prop}\label{reductiontoprojective}
    Let $g,n\geq 0$ and $\ell>1$. Let $\tilde{\rho}:\tPModl{\ell}{S_g^n}\ra \mathrm{GL_d}(\C)$ be a representation
    such that $(\id,1)$ acts by a scalar. Let $\rho:\PModl{\ell}{S_g^n}\ra \mathrm{PGL_d}(\C)$ denote the associated projective representation.

    Then $\tilde{\rho}$ is locally rigid if and only if $\rho$ is. More precisely, we have:
    $$\coh{\tPModl{\ell}{S_g^n}}{\ad\tilde{\rho}}=0
    \text{ if and only if }\coh{\PModl{\ell}{S_g^n}}{\ad\rho}=0.$$
\end{prop}

\begin{proof}
    First, we have an isomorphism of $\tPModl{\ell}{S_g^n}$-modules:
    $$\mathfrak{gl}_{\mathrm{d}}(\C)\simeq \mathfrak{sl}_{\mathrm{d}}(\C)\oplus \C.$$
    As $\tPModl{\ell}{S_g^n}^{\mathrm{ab}}$ is finite, we have $\coh{\tPModl{\ell}{S_g^n}}{\C}=0$.
    Hence: $$\coh{\tPModl{\ell}{S_g^n}}{\mathfrak{gl}_{\mathrm{d}}(\C)}=\coh{\tPModl{\ell}{S_g^n}}{\mathfrak{sl}_{\mathrm{d}}(\C)}.$$

    The kernel of $\tPModl{\ell}{S_g^n}\ra\PModl{\ell}{S_g^n}$ is a finite group $R$. The inflation restriction exact sequence
    \ref{inflationrestriction} is:
    $$0\lra \mathrm{H}^1(\PModl{\ell}{S_g^n},\mathfrak{sl}_{\mathrm{d}}(\C))\lra \mathrm{H}^1(\tPModl{\ell}{S_g^n},\mathfrak{sl}_{\mathrm{d}}(\C))
    \lra \mathrm{H}^1(R,\mathfrak{sl}_{\mathrm{d}}(\C)).$$
    But, as $R$ is finite, $\mathrm{H}^1(R,\mathfrak{sl}_{\mathrm{d}}(\C))=0$. Hence the result.
\end{proof}

We now show that proving local rigidity in $\PU{p}{q}$ is sufficient.

\begin{prop}\label{reductiontounitary}
    Let $\rho:G\ra \PU{p}{q}$ be a representation. Let $\overline{\rho}:G\lra \mathrm{PGL_{p+q}}(\C)$
    be the extension given by the inclusion $\PU{p}{q}\subset \mathrm{PGL_{p+q}}(\C)$.

    Then $\rho$ is locally rigid if and only if $\overline{\rho}$ is.
\end{prop}
    
\begin{proof}
    Let $\mathfrak{su}(p,q)\subset \mathfrak{sl}_n(\C)$ be the inclusion of real Lie algebras.
    One notices that there is a decomposition of $G$-modules (but not of Lie algebras):
    $$\mathfrak{su}(p,q)\oplus i\mathfrak{su}(p,q)= \mathfrak{sl}_n(\C).$$
    Now, $i\mathfrak{su}(p,q)$ is isomorphic to $\mathfrak{su}(p,q)$ as a $G$-module. Hence the result.
\end{proof}

We now turn to the proof of \Cref{removingthel} mentioned in \Cref{subsectionintroresults}.
The proof is based on a generalization to deformations of the following result. 

\begin{theorem}{\cite[2.5]{aramayonaRigidityPhenomenaMapping2016}}\label{quasiunipotent}
    Let $g\geq 3$ and $n\geq 0$. Then for any finite dimensional representation $\rho$ of $\PMod{S_g^n}$,
    and any Dehn twist $T_\gamma$, its image $\rho(T_\gamma)$ is quasi-unipotent.
\end{theorem}

We will denote by $\TU{d}$ the tangent space to the unitary group $\Ug{d}$. More precisely, it is the group
of matrices $V(\Id{d}+\epsilon A)$ with $V\in\Ug{d}$ and $A\in\mathfrak{u}_d$ (ie. $A$ is antisymmetric).

We will need the following two lemmas to generalize \Cref{quasiunipotent}.

\begin{lemma}\label{diagonalisable}
    Any matrix in $\TU{d}$ is conjugate in $\TU{d}$ to a diagonal matrix.
\end{lemma}
\begin{proof}
    Let $V(\Id{d}+\epsilon A)$ be an element of $\TU{d}$. Conjugating by elements of $\Ug{d}\subset\TU{d}$, we can assume that $V$ is diagonal,
    say $V=\mathrm{diag}(\lambda_1\Id{d_1},\dotsc,\lambda_n\Id{d_n})$ with $\lambda_i\neq \lambda_j$ for $i\neq j$. For $B\in\mathfrak{u}_d$,
    $(\Id{d}+\epsilon B)V(\Id{d}+\epsilon A)(\Id{d}+\epsilon B)^{-1}=V(\Id{d}+\epsilon(A+V^{-1}BV-B))$. Hence we may assume that $A$ is bloc
    diagonal of the form $A=\mathrm{diag}(A_1,\dotsc,A_n)$ with $A_i\in\mathfrak{u}_{d_i}$.
    Now for each $i$, $A_i=U_iD_iU_i^{-1}$ for $U_i\in\Ug{d_i}$ and $D_i$ diagonal. Setting $D=\mathrm{diag}(D_1,\dotsc,D_n)$,
    as $U=\mathrm{diag}(U_1,\dotsc,U_n)$ commutes with $V$, we see that it conjugates $V(\Id{d}+\epsilon A)$
    to the diagonal matrix $V(\Id{d}+\epsilon D)$.
\end{proof}

\begin{lemma}\label{finiteabelianization}
    For $g\geq 2$ and $n\geq 0$, $\tPMod{S_g^n}$ has finite abelianization.
\end{lemma}
\begin{proof}
    It is well known that in this range $\PMod{S_g^n}$ has finite abelianization (see for example \cite[5.1.2]{farbPrimerMappingClass2011}).
    Because abelianization is right exact, we get an exact sequence:
    \begin{equation*}
        \Z \lra \tPMod{S_g^n}^{\mathrm{ab}} \lra \PMod{S_g^n}^{\mathrm{ab}} \lra 1.
    \end{equation*}
    We only need to show that the map $\Z \ra \tPMod{S_g^n}^{\mathrm{ab}}$ is not injective. Let us assume by contradiction that it is.
    Then the map $\Q\ra \tPMod{S_g^n}^{\mathrm{ab}}\otimes\Q$ would be an isomorphism, so that the cocycle $\tau\in \mathrm{H}^2(\PMod{S_g^n};\Z)$
    corresponding to the extension $\tPMod{S_g^n}$ would be of torsion, as the corresponding extension by $\Q$ would split.
    However, the restriction of $\tau$ to $\mathrm{H}^2(\PMod{S_g^1};\Z)\simeq \Z$ by an inclusion $S_g^1\subset S_g^n$ is
    $4$ times a generator (see \cite{gilmerMaslovIndexLagrangians2013}) and is thus not a torsion element.
\end{proof}

\begin{theorem}\label{finiteorderdeformation}
    Let $g\geq 3$, $n\geq 0$ and $d\geq 0$, then for any representation $\rho:\tPMod{S_g^n}\ra \TU{d}$
    and any lift $T_\gamma$ of a Dehn twist to $\tPMod{S_g^n}$, its image $\rho(T_\gamma)$ has finite order
    equal to that of its projection to $\Ug{d}$.
\end{theorem}

\begin{proof}
    Let $\gamma$ be a simple closed curve and $T_\gamma\in\tPMod{S_g^n}$ a lift of the Dehn twist around $\gamma$. Denote by $S'$ the
    compact surface obtained by cutting $S_g^n$ along $\gamma$. It has a component $S''\subset S'$ of genus $g''\geq 2$.
    Let $\rho'':\tPMod{S''}\ra\TU{d}$ be the restriction of $\rho$. One may assume, by \Cref{diagonalisable}
    that $\rho(T_\gamma)=\rho''(T_\gamma)$ is diagonal. Let us denote for $\lambda\in \C^*$ and $\mu\in\C$ by $E_{\lambda,\mu}\subset\C[\epsilon]^d$
    the subspace generated by the coordinate vectors with diagonal coefficient $\lambda+\epsilon \mu$ in $\rho(T_\gamma)$.
    Clearily $\C[\epsilon]^d=\bigoplus_{\lambda,\mu}E_{\lambda,\mu}$. We will use the notation $E_\lambda=\bigoplus_\mu E_{\lambda,\mu}$.

    The twist $T_\gamma$ commutes with the image of $\tPMod{S''}$ in $\tPMod{S_g^n}$, so that\linebreak $\rho''(\tPMod{S''})$ commutes to $\rho(T_\gamma)$.
    Now the commutator of $\rho(T_\gamma)$ in $\TU{d}$ preserves each
    $\mathrm{ker}(\rho(T_\gamma)-(\lambda+\epsilon \mu)\Id{d})=E_{\lambda,\mu}+\epsilon E_\lambda$ and hence also
    each $\epsilon E_{\lambda,\mu}$. Let us fix $\lambda\in \C^*$ and $\mu\in \C$. The representation $\rho''$ preserves
    $\bigoplus_{\mu'\neq\mu} \epsilon E_{\lambda,\mu'}$ and thus acts on the quotient
    $E_{\lambda,\mu}+\epsilon E_\lambda/\bigoplus_{\mu'\neq\mu} \epsilon E_{\lambda,\mu'}$, that we identity with $E_{\lambda,\mu}$.
    Let $\rho''_{\lambda,\mu}:\tPMod{S''}\ra \Ug{E_{\lambda,\mu}}$ be this action.
    As by \Cref{finiteabelianization}, $\tPMod{S''}^\mathrm{ab}$ is finite, $\mathrm{det}(\rho''_{\lambda,\mu}):\tPMod{S''}\ra \C[\epsilon]^\times$
    factors through roots of unity in $\C^*$. So, if $d_{\lambda,\mu}=\dim E_{\lambda,\mu}>0$,
    $\mathrm{det}(\rho''_{\lambda,\mu})(T_\gamma)=(\lambda+\epsilon \mu)^{d_{\lambda,\mu}}$ must be a root of unity, ie. $\lambda$ is a root of unity
    and $\mu=0$. So $\rho(T_\gamma)=\rho''(T_\gamma)$ has finite order and the order is the same after quotient by $\epsilon$.
\end{proof}

\begin{prop}\label{removingthel}
    Let $\ell\geq 5$ be a prime number, $g\geq 3$ and $n\geq 0$.
    Let $\underline{\lambda}$ be any coloring of the boundary components of $S_g^n$.
    Then the $\SO$ quantum representation associated to $S_g^n$ and $\underline{\lambda}$ is locally rigid as a representation of $\PMod{S_g^n}$
    if and only if it is as a representation of $\PModl{\ell}{S_g^n}$.
\end{prop}

\begin{proof}
    Because of \Cref{reductiontounitary,reductiontoprojective}, we may restrict to deformations as unitary representations
    of the central extension.
    We need to show that any deformation $\hat{\rho}:\tPMod{S_g^n}\ra \TU{d}$ of the quantum representation $\rho$
    factors through $\tPModl{\ell}{S_g^n}$.
    Let $T_\gamma\in \tPModl{\ell}{S_g^n}$ be the lift of a Dehn twist. By \Cref{finiteorderdeformation},
    $\hat{\rho}(T_\gamma)$ has finite order equal to that of $\rho(T_\gamma)$.
    Let $c$ be the generator of the central extension. Then for any $\gamma$, $c=(c T_\gamma)T_\gamma^{-1}$ is the quotient of $2$
    commuting lifts of Dehn twists. Hence again $\hat{\rho}(c)$ has finite order equal to that of $\rho(c)$.
    Thus $\hat{\rho}$ factors through $\tPModl{\ell}{S_g^n}$.
\end{proof}

Finally, we mention that embedding the representation in a larger linear group does not change the space of deformations.

\begin{prop}\label{largergl}
    Let $\rho:\tPModl{\ell}{S}\lra \mathrm{GL}_d(\C)$ be a representation such that $\rho((\id,1))$ is non-trivial.
    Let $\rho':\tPModl{\ell}{S}\lra \mathrm{GL}_{d+N}(\C)$ be the post-composition of the representation with the inclusion
    $\mathrm{GL}_d(\C)\ra \mathrm{GL}_{d+N}(\C)$. Then the induced morphism:
    $$\coh{\tPModl{\ell}{S}}{\ad \rho}\lra \coh{\tPModl{\ell}{S}}{\ad \rho'}$$
    is an isomorphism.
\end{prop}
The hypothesis on $\rho((\id,1))$ is always verified for representations coming from the $\SO$ TQFT.

\begin{proof}
    As a representation of $\tPModl{\ell}{S}$, $\ad \rho'$ has a decomposition:
    $$\mathfrak{gl}_{d+N}=\mathfrak{gl}_{d}\oplus \rho^{\oplus N}\oplus (\rho^*)^{\oplus N}\oplus \mathrm{1}^{\oplus N^2}$$
    where $1$ denotes the trivial representation.
    Now, as $\tPModl{\ell}{S}^\mathrm{ab}$ is finite: 
    $$\coh{\tPModl{\ell}{S}}{\mathrm{1}}=0.$$
    As $(\id,1)$ is central and $\rho((\id,1))$ is non-trivial, by \Cref{centerkills}:
    $$\coh{\tPModl{\ell}{S}}{\rho}=0\text{ and }\coh{\tPModl{\ell}{S}}{\rho^*}=0.$$
\end{proof}

\section{Proof of the main Lemma}\label{sectionmainlemma}

From this point onwards, we will use the following notation for the adjoint representations.
For $(S,\underline{\lambda})$ a colored surface, we will denote $\ad\Nu(S,\underline{\lambda})$ or even just $\ad\Nu(S)$
the adjoint of the representation of $\tPMod{S}$ on $\Nu(S,\underline{\lambda})$, ie. the space $\End{\Nu(S,\underline{\lambda})}$
with action of $\tPMod{S}$ by conjugation.
This notation will be useful when dealing with 
mapping class groups of subsurfaces.

Let $(S,\underline{\lambda})$ be a colored surface, $a_1,\dotsc,a_n$ some disjoint oriented {\sccs} on $S$
and $\und{\mu}\in\Lambda^n$ some colors.
We will denote by $(S_{a_1,\dotsc,a_n},\und{\mu},\und{\lambda})$,
or even just by $(S_{a_1,\dotsc,a_n},\mu_1,\dotsc,\mu_n)$,
the colored surface obtained by cutting $S$ along $a_1,\dotsc,a_n$ and
coloring the new boundary components on either sides of $a_i$ with $\mu_i$. See \Cref{cut_example} for an example.

\begin{figure}
    \def\svgwidth{0.8\linewidth}
    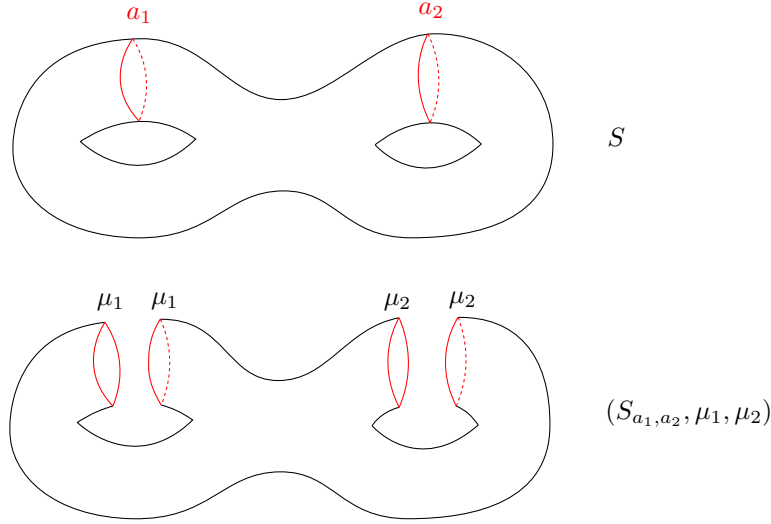
    \caption{Surfaces $S$ and $(S_{a_1,a_2},\mu_1,\mu_2)$}
    \label[figure]{cut_example}
\end{figure}

One of the central technical results of this paper is the following lemma.

\begin{lemma}\label{mainlemma}
    Let $\Nu$ be a modular functor. Let $\ell>1$ be a level for $\Nu$.
    
    Let $(S,\und{\lambda})$ be a connected colored surface, and $a_1$, $a_2$ disjoint {\sccs} on $S$ such that $\PMod{S}$
    is generated by the stabilizer of $a_1$ and the stabilizer of $a_2$.
    Let $c$ be a set of disjoint simple closed curves separating $S$ into $2$ components, containing $a_1$ and $a_2$ respectively.
    
    We assume that $\Nu$ is irreducible on the $2$ components of $S_c$ and those of $S_{a_1,a_2}$.

    Assume there exist colors $\underline{\lambda_c}$ such that for all $\mu_1,\mu_2\in \Lambda$:
    $$\Nu(S_{a_1,a_2},\mu_1,\mu_2)\neq 0\Lra \Nu(S_{a_1,a_2,c},\mu_1,\mu_2,\underline{\lambda_c})\neq 0.$$
    Then the natural map:
    $$\coh{\PModl{\ell}{S}}{\ad\Nu(S)}\ra \coh{\PModl{\ell}{S_{a_1}}}{\ad\Nu(S)}\oplus \coh{\PModl{\ell}{S_{a_2}}}{\ad\Nu(S)}$$
    is injective.
\end{lemma}

\begin{figure}
    \def\svgwidth{0.8\linewidth}
\begingroup%
  \makeatletter%
  \providecommand\color[2][]{%
    \errmessage{(Inkscape) Color is used for the text in Inkscape, but the package 'color.sty' is not loaded}%
    \renewcommand\color[2][]{}%
  }%
  \providecommand\transparent[1]{%
    \errmessage{(Inkscape) Transparency is used (non-zero) for the text in Inkscape, but the package 'transparent.sty' is not loaded}%
    \renewcommand\transparent[1]{}%
  }%
  \providecommand\rotatebox[2]{#2}%
  \newcommand*\fsize{\dimexpr\f@size pt\relax}%
  \newcommand*\lineheight[1]{\fontsize{\fsize}{#1\fsize}\selectfont}%
  \ifx\svgwidth\undefined%
    \setlength{\unitlength}{253.71149258bp}%
    \ifx\svgscale\undefined%
      \relax%
    \else%
      \setlength{\unitlength}{\unitlength * \real{\svgscale}}%
    \fi%
  \else%
    \setlength{\unitlength}{\svgwidth}%
  \fi%
  \global\let\svgwidth\undefined%
  \global\let\svgscale\undefined%
  \makeatother%
  \begin{picture}(1,0.41976777)%
    \lineheight{1}%
    \setlength\tabcolsep{0pt}%
    \put(0,0){\includegraphics[width=\unitlength,page=1]{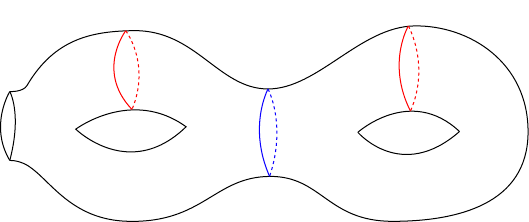}}%
    \put(0.50221686,0.04247293){\color[rgb]{0.05098039,0,1}\makebox(0,0)[lt]{\lineheight{1.25}\smash{\begin{tabular}[t]{l}$c$\end{tabular}}}}%
    \put(0.22371322,0.38943688){\color[rgb]{1,0,0}\makebox(0,0)[lt]{\lineheight{1.25}\smash{\begin{tabular}[t]{l}$a_1$\end{tabular}}}}%
    \put(0.75672128,0.38943688){\color[rgb]{1,0,0}\makebox(0,0)[lt]{\lineheight{1.25}\smash{\begin{tabular}[t]{l}$a_2$\end{tabular}}}}%
  \end{picture}%
\endgroup%

    \caption{Surface $S$ and curves $c$, $a_1$, $a_2$ as in \Cref{mainlemma}.}
    \label[figure]{lemma_example}
\end{figure}
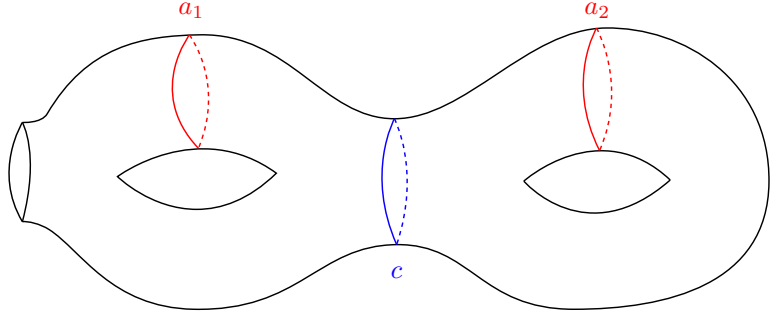

\begin{remark}
    When applying \Cref{mainlemma}, we will usually not give details on why $\PMod{S}$
    is generated by the stabilizer of $a_1$ and the stabilizer of $a_2$, as it will easily follow from well known generating sets for $\PMod{S}$.
    For details on these generators, we refer the reader to \cite[4.4 and 9.3]{farbPrimerMappingClass2011}.
\end{remark}

The proof of the lemma relies on the following group cohomology results.

\begin{prop}[MV-sequence, {\cite[chp. II (7.7), (7.8)]{brownCohomologyGroups1982}}] \label{mayervietoris}
    Let $G_1$, $G_2$, $A$ be groups with inclusions $i_1:A\lra G_1$, $i_2:A\lra G_2$.
    Let $G=G_1*_A G_2$ be their amalgamated sum. Let $M$ be a $G$-module.
    Then one has an exact sequence of $G$-modules:
    $$0\lra \Z[G/A]\lra \Z[G/G_1]\oplus\Z[G/G_2] \
    \xlra{\begin{pmatrix} 1 & -1 \end{pmatrix}} \Z\lra 0$$
    which induces a long exact sequence in cohomology:
    $$\ra \mathrm{H}^n(G,M) \ra \mathrm{H}^n(G_1,M)\oplus \mathrm{H}^n(G_2,M) \ra \mathrm{H}^n(A,M)\ra \mathrm{H}^{n+1}(G,M)\ra.$$
\end{prop}

\begin{prop}[IR-sequence, {\cite[VII, Proposition 4]{serreLocalFields2013}}] \label{inflationrestriction}
    Let $G$ be a group, $R$ a normal subgroup and $M$ a $G$-module. Then we have the following exact sequence:
    $$0\lra \mathrm{H}^1(G/R,M^R)\lra \mathrm{H}^1(G,M)\lra \mathrm{H}^1(R,M).$$
    Where $M^R$ denotes the set of elements of $M$ fixed by $R$.
\end{prop}

The idea behind the proof of \Cref{mainlemma} is the following. If $\rho$ is a representation of
$\PModl{\ell}{S}$ that is rigid when restricted to $\PModl{\ell}{S_{a_1}}$ and $\PModl{\ell}{S_{a_2}}$,
one could hope that, as these groups generate $\PModl{\ell}{S}$, $\rho$ is rigid.
However this is not the case in general. The defect of rigidity of such a $\rho$
can be controlled by the Mayer-Vietoris and Inflation-Restriction sequences (see step \textbf{(1)} below).
In the situation of the Lemma, we use the technical assumption on $c$ and the irreducibility of $\Nu$
to control this defect.

\begin{proof}[Proof of \Cref{mainlemma}.]
    \textbf{(1) Reformulation}.
    For $i\in\{1,2\}$, let $\Gamma_i=\PModl{\ell}{S,[a_i]}$ denote the stabilizer of $a_i$, and let $\Gamma_{12}=\PModl{\ell}{S,[a_1],[a_2]}$ their intersection.
    Let $G=\Gamma_1*_{\Gamma_{12}}\Gamma_2$ be their amalgamated sum.
    
    As by assumption the stabilizers generate $\Gamma=\PModl{\ell}{S}$, we have $\Gamma = G/R$ for a normal subgroup $R$.
    Let $S_c= S_1 \sqcup S_2$ such that $a_i\subset S_i$.
    For $\Gamma_1^c = \PModl{\ell}{S_2}$ and $\Gamma_2^c = \PModl{\ell}{S_1}$, one has $\Gamma_i^c\subset \Gamma_i$, which justifies the notation.
    As $\Gamma_1^c$ and $\Gamma_2^c$ commute in $\Gamma$, we have an inclusion $[\Gamma_1^c,\Gamma_2^c]\subset R$.
    
    From Mayer-Vietoris sequence (\Cref{mayervietoris}) and the Inflation Restriction sequence (\Cref{inflationrestriction}),
    we get the following diagram with exact vertical and horizontal sequences.
    \[\begin{tikzcd}[column sep=small]
        && 0 \\
        && {\mathrm{H}^1(\Gamma,M)} \\
        {M^{\Gamma_1}\oplus M^{\Gamma_2}} & {M^{\Gamma_{12}}} & {\mathrm{H}^1(G,M)} & {\mathrm{H}^1(\Gamma_1,M)\oplus \mathrm{H}^1(\Gamma_2,M)} \\
        && {\mathrm{Hom}(R,M)}
        \arrow[from=1-3, to=2-3]
        \arrow[from=2-3, to=3-3]
        \arrow["f"', from=3-3, to=4-3]
        \arrow[from=3-3, to=3-4]
        \arrow["\delta", from=3-2, to=3-3]
        \arrow[from=3-1, to=3-2]
    \end{tikzcd}\]
    Here, $M = \ad\Nu(S)$. From the diagram we can see that if $f$ is injective on the image of $\delta$,
    then $\mathrm{H}^1(\Gamma,M)$ injects into $\mathrm{H}^1(\Gamma_1,M)\oplus \mathrm{H}^1(\Gamma_2,M)$.
    
    Note that $\Gamma_i = \PModl{\ell}{S_{a_i}}/\langle T_{a_i} \rangle$, where $T_{a_i}$ is the Dehn twist along $a_i$.
    As $\langle T_{a_i} \rangle$ is finite, one has $\mathrm{Hom}(\langle T_{a_i} \rangle,M)=0$.
    Thus from the Inflation-Restriction sequence (\Cref{inflationrestriction}):
    $$0\lra \mathrm{H}^1(\Gamma_i,M)\lra \mathrm{H}^1(\PModl{\ell}{S_{a_i}},M)\lra \mathrm{Hom}(\langle T_{a_i} \rangle,M)$$
    we can see that $\mathrm{H}^1(\Gamma_i,M)=\mathrm{H}^1(\PModl{\ell}{S_{a_i}},M)$.

    Hence, if we can show that $f$ is injective on the image of $\delta$, we are done.

    \textbf{(2) Computation of $f\circ\delta$}. To compute $\delta$, one has to choose acyclic resolutions of $\Z$, $\Z[G/G_1]$,\dots
    We will choose the resolutions giving the canonical description of chains in group cohomology (see \cite[VII.3]{serreLocalFields2013}).
    With these resolutions, the chains for the $G$-module $\Z[G/H]$ are given by:
    $$C^n(H,M)=\mathrm{Hom}_{\mathrm{Set}}(H^n,M).$$
    The first $2$ differentials are:
    \[\begin{array}{rlll}
        \partial_0:&M&\lra& \mathrm{Hom}_{\mathrm{Set}}(H,M) \\
                   &m&\longmapsto& (g\mapsto g\cdot m - m) \\
        \partial_1:&\mathrm{Hom}_{\mathrm{Set}}(H,M)&\lra& \mathrm{Hom}_{\mathrm{Set}}(H^2,M) \\
                   &\varphi&\longmapsto& ((g_1,g_2)\mapsto g_1\cdot \varphi(g_2) - \varphi(g_1g_2)+\varphi(g_1)).
    \end{array}\]

    Let $m\in M^{\Gamma_{12}}\subset M=C^0(\Gamma_{12},M)$. We can lift $m$ to $m+0\in M\oplus M=C^0(\Gamma_1,M)\oplus C^0(\Gamma_2,M)$.
    Now:
    $$\partial_0(m+0) =\varphi+0 \in \mathrm{Hom}(\Gamma_1,M)\oplus \mathrm{Hom}(\Gamma_1,M)=C^1(\Gamma_1,M)\oplus C^1(\Gamma_2,M)$$
    where $\varphi(g_1) = g_1\cdot m - m$. Hence $\delta(m)\in \mathrm{Hom}(G,M)$ is the unique cocycle $\psi:G\lra M$ such that
    $\psi(g_1)=\varphi(g_1)$ for all $g_1\in\Gamma_1$ and $\psi(g_2) = 0$ for all $g_2\in \Gamma_2$.

    Let $g_1\in \Gamma_1^c$ and $g_2\in \Gamma_2^c$. We want to compute $f(\delta(m))(g_1g_2g_1^{-1}g_2^{-1})=\psi(g_1g_2g_1^{-1}g_2^{-1})$.
    From the cocycle condition:
    \[\begin{array}{ll}
    \psi(g_1g_2g_1^{-1}g_2^{-1}) &= \psi(g_1) + g_1\cdot\psi(g_2) + g_1g_2\cdot \psi(g_1^{-1})+g_1g_2g_1^{-1}\cdot \psi(g_2^{-1}) \\
                    &= \psi(g_1) + g_1\cdot\psi(g_2) - g_2\cdot \psi(g_1) - \psi(g_2) \\
                    &= g_1\cdot m  - m - g_2g_1\cdot m + g_2\cdot m.
    \end{array}\]
    We used $\psi(h^{-1})=-h^{-1}\cdot\psi(h)$, the fact that $g_1$ and $g_2$ commute in $\Gamma$, $\psi(g_1) = g_1\cdot m - m$
    and $\psi(g_2)=0$.
    Now $M$ is by definition $\ad\Nu(S)$. More precisely, $M=\End{\Nu(S)}$ with the action $g\cdot m = \rho(g)m\rho(g^{-1})$,
    where $\rho:\tPMod{S} \ra \GL{\Nu(S)}$ is the quantum representation. In what follows we will simply denote the
    endomorphism $\rho(g)\in \GL{\Nu(S)}$ by $g$. We will also use the notation $[\cdot,\cdot]$ for the
    Lie bracket $[u,v]=u\circ v- v\circ u$ in $M=\End{\Nu(S)}$. Let us continue our computation.
    \[\begin{array}{ll}
    \psi(g_1g_2g_1^{-1}g_2^{-1}) &= g_1mg_1^{-1} - m - g_2g_1mg_1^{-1}g_2^{-1} + g_2mg_2^{-1} \\
                    &= (g_1mg_2 - mg_1g_2 - g_2g_1m + g_2mg_1)g_1^{-1}g_2^{-1} \\
                    &= -[ g_1,[ g_2,m]] g_1^{-1}g_2^{-1}.
    \end{array}\]
    Hence, if $m$ is in the kernel of $f\circ \delta$, for all $g_1\in \Gamma_1^c$ and $g_2\in \Gamma_2^c$,
    $[ g_1,[ g_2,m]] =0$.

    \textbf{(3) Reduction to \Cref{linearalgebra}}.
    From the gluing axiom applied along $c$, we have a decomposition:
    $$\Nu(S)=\bigoplus_{\underline{\lambda}} \Nu(S_{1},\underline{\lambda})\otimes \Nu(S_{2},\underline{\lambda}).$$
    Let $M' = \mathrm{End}(S_{1},\underline{\lambda_c})\otimes \mathrm{End}(S_{2},\underline{\lambda_c})$.
    As $M=\mathrm{End}(\Nu(S))$, we can decompose $M$ as:
    $$M=\bigoplus_{\underline{\lambda}}\bigoplus_{\underline{\mu}} \Hom{\Nu(S_{1},\underline{\lambda})}{\Nu(S_{1},\underline{\mu})}
    \otimes \Hom{\Nu(S_{2},\underline{\lambda})}{\Nu(S_{2},\underline{\mu})}.$$
    One of these summands is $M'$. Thus we have a natural projection $M\lra M'$.
    Moreover, as elements of $\Gamma_1^c$ and $\Gamma_2^c$ preserve this
    decomposition of $M$, their actions commute with this projection $M\lra M'$.

    Hence if we denote by $m'$ the image of $m$ in $M'$, provided $m$ is in the kernel of $f\circ \delta$,
    we have $[ g_1,[ g_2,m']] =0$ in $M'$ for any $g_i\in \Gamma_i^c$, $i\in\{1,2\}$.

    As $\Nu$ is irreducible on $S_1$, the image of $\C[\Gamma_1^c]$ in $M'$ is 
    $E_2 = \id\otimes \mathrm{End}(S_{2},\underline{\lambda_c})$. Similarly, the image of $\C[\Gamma_2^c]$ in $M'$ is 
    $E_1 = \mathrm{End}(S_{1},\underline{\lambda_c})\otimes\id$.

    Let $m$ be in the kernel of $f\circ \delta$. Then, from the gluing axiom applied to $S$ along $a_1$ and $a_2$ we have the decomposition:
    $$\Nu(S)=\bigoplus_{\mu_1,\mu_2} \Nu(S_{a_1,a_2},\mu_1,\mu_2).$$
    As $\Nu$ is irreducible on the components of $S_{a_1,a_2}$ and as $m$ commutes to $\Gamma_{12}$, $m$ decomposes as:
    $$m=\sum_{\mu_1,\mu_2} c_{\mu_1\mu_2}\id_{\Nu(S_{a_1,a_2},\mu_1,\mu_2)}$$
    with $c_{\mu_1\mu_2}\in\C$.
    Thus $m'$ decomposes as:
    $$m'=\sum_{\mu_1,\mu_2} c_{\mu_1\mu_2}\id_{\Nu(S_{1,a_1},\underline{\lambda_c},\mu_1)}\otimes \id_{\Nu(S_{2,a_2},\underline{\lambda_c},\mu_2)}.$$

    Now, applying \Cref{linearalgebra} with $V_1=\Nu(S_{1},\underline{\lambda_c})$
    and $V_2=\Nu(S_{2},\underline{\lambda_c})$, we get that there exists
    $a$ and $b$ such that whenever $\Nu(S_{a_1,a_2,c},\mu_1,\mu_2,\underline{\lambda_c})\neq 0$, one has $c_{\mu_1,\mu_2}=a_{\mu_1}+b_{\mu_2}$.
    As, by hypothesis, $\Nu(S_{a_1,a_2,c},\mu_1,\mu_2,\underline{\lambda_c})\neq 0$ whenever $\Nu(S_{a_1,a_2},\mu_1,\mu_2)\neq 0$, this just means
    that $m$ is in the image of $M^{\Gamma_1}\oplus M^{\Gamma_2}\lra M^{\Gamma_{12}}$, ie. in the kernel of $\delta$.

    Thus $f$ is injective on the image of $\delta$, and the lemma is proved.
\end{proof}

\begin{lemma}\label{linearalgebra}
    Let $V_i = \bigoplus_{\mu_i} V_i(\mu_i)$, $i=1,2$, be vector spaces. Set $M_i=\{\mu_i\mid V_i(\mu_i)\neq 0\}$. Let:
    $$u=\sum_{\mu_1,\mu_2}c_{\mu_1\mu_2}\id_{V_1(\mu_1)}\otimes \id_{V_2(\mu_2)}\in\End{V_1\otimes V_2}$$
    with $c_{\mu_1\mu_2}\in \C$. Let $E_1 = \End{V_1}\otimes \id$ and $E_2 = \id \otimes \End{V_2}$.
    Then if:
    $$[ E_2,[ E_1,u]] =0$$
    there exist $(a_\mu)_{\mu\in M_1}$, $(b_\mu)_{\mu\in M_2}$ such that for all $\mu_i\in M_i$, $i=1,2$:
    $$c_{\mu_1\mu_2} = a_{\mu_1} + b_{\mu_2}.$$
    Here $[\cdot,\cdot]$ denotes the Lie bracket.
\end{lemma}

\begin{proof}
    Let $\lambda_i,\nu_i\in M_i$ for $i=1,2$. Let $m_{\lambda_i\nu_i} \in \Hom{V_i(\lambda_i)}{V_i(\nu_i)}$ for $i=1,2$.
    Define $x_1 = m_{\lambda_1\nu_1}\otimes \id_{V_2}$ and $x_2 = \id_{V_1}\otimes m_{\lambda_2\nu_2}$.
    Now:
    \[\begin{array}{lcl}
        [ x_1,u]  &=& x_1\circ u - u\circ x_1 \\
                &=& \sum_{\mu_2} c_{\nu_1\mu_2}m_{\lambda_1\nu_1}\otimes\id_{V_2(\mu_2)} - c_{\lambda_1\mu_2}m_{\lambda_1\nu_1}\otimes\id_{V_2(\mu_2)} \\
                &=& m_{\lambda_1\nu_1}\otimes \big(\sum_{\mu_2} (c_{\nu_1\mu_2} - c_{\lambda_1\mu_2})\id_{V_2(\mu_2)}\big).
    \end{array}
    \]
    And:
    \[\begin{array}{lcl}
        [ x_2,[ x_1,u]]  &=& x_2\circ [ x_1,u] - [ x_1,u] \circ x_2 \\
                &=& m_{\lambda_1\nu_1}\otimes (c_{\nu_1\nu_2} - c_{\lambda_1\nu_2})m_{\lambda_2\nu_2} \\
                &&-m_{\lambda_1\nu_1}\otimes (c_{\nu_1\lambda_2} - c_{\lambda_1\lambda_2})m_{\lambda_2\nu_2} \\
                &=&(c_{\nu_1\nu_2} - c_{\lambda_1\nu_2} - c_{\nu_1\lambda_2} + c_{\lambda_1\lambda_2})m_{\lambda_1\nu_1}\otimes m_{\lambda_2\nu_2}.
    \end{array}
    \]
    As the last line must be $0$, one has for all $\lambda_i,\nu_i\in M_i$, $i=1,2$:
    $$c_{\nu_1\nu_2} - c_{\lambda_1\nu_2} - c_{\nu_1\lambda_2} + c_{\lambda_1\lambda_2} = 0.$$
    One can check that for fixed $\nu_i\in M_i$ for $i=1,2$, $a_{\mu_1} = c_{\mu_1\nu_2}$ and $b_{\mu_2} = c_{\nu_1\mu_2} - c_{\nu_1\nu_2}$
    verify the claim of the lemma.
\end{proof}

\section{\texorpdfstring{Proof of \Cref{induction}}{Proof of Theorem \ref{induction}}}\label{sectioninduction}

The idea behind the proof of \Cref{induction} is to proceed by induction on the genus and number of marked points
by repeated use of \Cref{mainlemma}. 

\begin{lemma}[Center Kills]\label{centerkills}
    Let $G$ be a group and $M$ be a $\C[G]$-module.
    Let $Z\subset G$ be a finite central subgroup of $G$ such that its action on $M$ is given by a non-trivial character $\chi:Z\lra \C^*$.
    Then $\mathrm{H}^1(G,M)=0$.
\end{lemma}

\begin{proof}
    As $Z$ is central, it is a normal subgroup of $G$. Moreover, as $\chi$ is non-trivial, we have $M^Z=0$.
    Hence, applying \Cref{inflationrestriction}, we get an exact sequence:
    $$0\lra \mathrm{H}^1(G/Z,0)\lra \mathrm{H}^1(G,M)\lra \mathrm{H}^1(Z,M).$$
    Now, as $Z$ is finite, any $\C[Z]$-module is projective. Hence $\mathrm{H}^1(Z,M)=0$.

    Thus $\mathrm{H}^1(G,M)=0$.
\end{proof}

\begin{coro}\label{crossterms}
    Let $S$ be a compact surface with $\partial S\neq \varnothing$.
    Let $\underline{\lambda}$ and $\underline{\mu}$ be two colorings of $\partial S$.
    Assume $\Nu$ is a modular functor of level $\ell$ satisfying assumption \textbf{(III)}.
    If $\underline{\lambda}\neq \underline{\mu}$, then:
    $$\coh{\PModl{\ell}{S}}{\Nu(S,\underline{\lambda})^*\otimes \Nu(S,\underline{\mu})}=0.$$
\end{coro}

\begin{proof}
    Let $T_\gamma$ be a Dehn twist along a boundary component of $S$ on which $\underline{\lambda}$ and $\underline{\mu}$ differ.
    Then, by \textbf{(III)}, $Z=\langle T_\gamma \rangle$ acts non-trivially by scalars on $\Nu(S,\underline{\lambda})^*\otimes \Nu(S,\underline{\mu})$.
    As $Z$ is central and finite in $\PModl{\ell}{S}$, the result is a consequence of \Cref{centerkills}.
\end{proof}

\begin{coro}\label{bettermainlemma}
    We assume $\Nu$, $\ell$, $S$, $a_i$, $c$, $\underline{\lambda_c}$ satisfy the hypotheses of \Cref{mainlemma}.

    We also assume $\Nu$ verifies \textbf{(III)}.

    Then for $i=1,2$:
    $$\coh{\PModl{\ell}{S_{a_i}}}{\ad\Nu(S)}=\bigoplus_{\mu_i}\coh{\PModl{\ell}{S_{a_i}}}{\ad\Nu(S_{a_i},\mu_i)}.$$
    In particular, we have an injective map:
    $$\coh{\PModl{\ell}{S}}{\ad\Nu(S)}\lra \bigoplus\limits_{i,\mu_i}\coh{\PModl{\ell}{S_{a_i}}}{\ad\Nu(S_{a_i},\mu_i)}.$$
\end{coro}

\begin{proof}
    Fix $i\in\{1,2\}$. From axiom \textbf{(G)}, we have:
    $$\Nu(S)=\bigoplus_{\mu_i}\Nu(S_{a_i},\mu_i).$$
    And thus:
    $$\ad\Nu(S)=\bigoplus_{\mu_i,\mu_i'}\Nu(S_{a_i},\mu_i')^*\otimes \Nu(S_{a_i},\mu_i).$$
    Now apply \Cref{crossterms} to conclude.
\end{proof}

\begin{lemma}[Künneth]\label{kunneth}
    Let $G_1$ and $G_2$ be $2$ groups. For $i=1,2$, let $M_i$ be a $\C[G_i]$-module. Then we have an isomorphism of graded vector spaces:
    $$H^*(G_1,M_1)\otimes H^*(G_2,M_2)\simeq H^*(G_1\times G_2,M_1\otimes M_2).$$
\end{lemma}

\begin{proof}
    For $i=1,2$, let $X_i$ be a CW-complex modeling the classifying space of $G_i$, and let $\mathcal{L}_i$
    be the local system of coefficients on $X_i$ corresponding to $M_i$.
    Then $X_1\times X_2$ is a CW-complex modeling $G_1\times G_2$, and $M_1\otimes M_2$ corresponds to the local system of coefficients
    $\mathcal{L}_1\otimes\mathcal{L}_2$. Now, as the cells of $X_1\times X_2$ are products of cells, we have an isomorphism of cellular complexes:
    $$C^*(X_1\times X_2;\mathcal{L}_1\otimes\mathcal{L}_2)\simeq C^*(X_1;\mathcal{L}_1)\otimes C^*(X_2;\mathcal{L}_2).$$
    Hence, as these are complexes of vector spaces, the homological Künneth formula yields an isomorphism:
    $$H^*(X_1\times X_2;\mathcal{L}_1\otimes\mathcal{L}_2)\simeq H^*(X_1;\mathcal{L}_1)\otimes H^*(X_2;\mathcal{L}_2).$$
\end{proof}

From now on, we will use the following corollary of the Künneth formula without mentioning it.
It enables us to work component by component when computing local rigidity.

\begin{coro}
    Let $S=S_1\sqcup S_2$ be a surface with $2$ connected components. For $i=1,2$, let $\rho_i$ be a an irreducible complex representation
    of $\tPModl{\ell}{S_i}$. Then we have an isomorphism:
    $$H^1(\tPModl{\ell}{S},\ad (\rho_1\otimes \rho_2))\simeq H^1(\tPModl{\ell}{S_1},\ad \rho_1)\oplus H^1(\tPModl{\ell}{S_2},\ad \rho_2).$$
\end{coro}

\begin{proof}
    For $i=1,2$, let $G_i=\tPModl{\ell}{S_i}$ and $M_i= \ad \rho_i$. Now $\tPModl{\ell}{S}=G_1\times G_2$
    and $\ad (\rho_1\otimes \rho_2)=\ad \rho_1\otimes \ad\rho_2$.
    Hence, applying \Cref{kunneth}, we have:
    \begin{multline*}
        H^1(G_1\times G_2,M_1\otimes M_2) \\
        \simeq H^0(G_1,M_1)\otimes H^1(G_2,M_2)\oplus H^1(G_1,M_1)\otimes H^0(G_2,M_2)
    \end{multline*}
    But as $\rho_i$ is irreducible, $H^0(G_i,\ad \rho_i)=\C$. Hence the result.
\end{proof}

We can now proceed with the proof of the theorem.

\begin{proof}[Proof of \Cref{induction}.]
    Since $g\geq 3$, by \Cref{removingthel}, we only need to prove local rigidity as representations of $\PModl{\ell}{S_g^n}$.

    \textbf{(1)} Let us first reduce to the case $g=3$.
    
    Let $g\geq 4$. Assume that for all $g'\in\{3,\dotsc,g-1\}$, $n'\geq 0$
    and $\underline{\lambda}'$ coloring of $\partial S_{g'}^{n'}$:
    $$\coh{\PModl{\ell}{S_{g'}^{n'}}}{\ad \Nu(S_{g'}^{n'},\underline{\lambda}')}=0.$$

    Let $n\geq 0$ and $\underline{\lambda}$ be any coloring of $\partial S_g^n$.
    We assume that $\Nu(S_g^n,\underline{\lambda})\neq 0$. Otherwise, the result is trivial.

    \begin{figure}
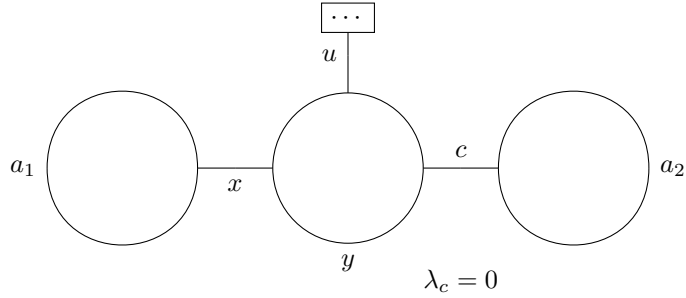

        \ctikzfig{graph_for_induction_on_g}
        \caption{Graph of a pair of pants decomposition.
        Edges correspond to \sccs, vertices to pairs of pants. 
        See the end of \Cref{modularfunctorbases} for more details.}
        \label[figure]{graph_for_induction_on_g}
    \end{figure}

    Then $S_g^n$ has a pair of pants decomposition with associated graph of the form described on \Cref{graph_for_induction_on_g}.
    Let us check that the cuts $a_1$, $a_2$ and $c$ verify the hypotheses of \Cref{mainlemma} for $\lambda_c=0$.
    Let $\mu_1, \mu_2\in \Lambda$. We show that:
    $$\Nu(S_{a_1,a_2,c},\mu_1,\mu_2,\lambda_c)\neq 0.$$
    Let $\nu_u$ be a color that appears on the edge $u$ in an admissible coloring of the graph.
    Then, by \textbf{(I)}, there exists $\nu_y\in\Lambda$ such that $\Nu(S_0^3,\nu_u,\nu_y,\nu_y)\neq 0$.
    Hence, there exists an admissible coloring of the graph with $u$ colored with $\nu_u$, $y$ colored with $\nu_y$,
    $x$ and $c$ colored with $0$, $a_1$ colored with $\mu_1$ and $a_2$ colored with $\mu_2$.
    Thus $\Nu(S_{a_1,a_2,c},\mu_1,\mu_2,\lambda_c)\neq 0$.

    Now, as any component of $S_c$ or $S_{a_1,a_2}$ has genus at least $1$,
    $\Nu$ is irreducible on these for any colorings.

    Hence the hypotheses of \Cref{mainlemma} are verified. Moreover, as $\Nu$ is assumed to verify \textbf{(III)},
    the conclusion of \Cref{bettermainlemma} holds, ie. the map:
    $$\coh{\PModl{\ell}{S}}{\ad\Nu(S)}\lra \bigoplus\limits_{i,\mu_i}\coh{\PModl{\ell}{S_{a_i}}}{\ad\Nu(S_{a_i},\mu_i)}$$
    is injective. But, by the induction hypothesis we made, the right hand side is $0$.
    Hence $\coh{\PModl{\ell}{S}}{\ad\Nu(S)}=0$.

    \textbf{(2)} The case $g=3$ remains. Let $n>4$. Assume that for all $n'\in\{0,\dotsc,n-1\}$ and any coloring $\underline{\lambda}'$
    of $\partial S_3^{n'}$, $\coh{\PModl{\ell}{S_{3}^{n'}}}{\Nu(S_{3}^{n'},\underline{\lambda}')}=0$.

    Let $\underline{\lambda}$ be any coloring of $\partial S_3^n$.
    Again, we assume $\Nu(S_3^n,\underline{\lambda})\neq 0$. Otherwise, the result is trivial.

    \begin{figure}
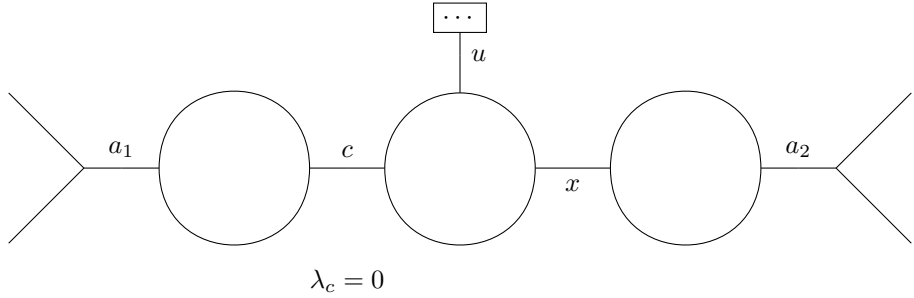

        \ctikzfig{graph_for_induction_on_n}
        \caption{Graph of a pair of pants decomposition.
        Edges correspond to \sccs, legs to boundary components and vertices to pairs of pants.
        See the end of \Cref{modularfunctorbases} for more details.}
        \label[figure]{graph_for_induction_on_n}
    \end{figure}

    Then $S_3^n$ has a pair of pants decomposition with associated graph of the form described on \Cref{graph_for_induction_on_n}.
    Let us verify that the cuts $a_1$, $a_2$ and $c$ verify the hypotheses of \Cref{mainlemma} for $\lambda_c=0$.
    Let $\mu_1, \mu_2\in \Lambda$. We show that:
    $$\Nu(S_{a_1,a_2,c},\mu_1,\mu_2,\lambda_c)\neq 0.$$
    Let $\nu_u$ be a color that appears on the edge $u$ in an admissible coloring of the graph.
    Then, as above, by \textbf{(I)}, there exists an admissible coloring of the graph with $u$ colored with $\nu_u$,
    $x$ and $c$ colored with $0$, $a_1$ colored with $\mu_1$ and $a_2$ colored with $\mu_2$.
    Thus $\Nu(S_{a_1,a_2,c},\mu_1,\mu_2,\lambda_c)\neq 0$.

    Now, every component of $S_c$ or $S_{a_1,a_2}$ has genus at least $1$ or is homeomorphic to $S_0^3$.
    Since $\Nu$ has dimension at most $1$ on any coloring of $S_0^3$,
    $\Nu$ is irreducible on the components of $S_c$ and $S_{a_1,a_2}$ for any colorings.

    As above, by \Cref{bettermainlemma} and the induction hypothesis, we have:
    $$\coh{\PModl{\ell}{S}}{\ad\Nu(S)}=0.$$

    We are thus reduced to the cases where $g=3$ and $n\in\{0,1,2,3\}$.
\end{proof}

\section{\texorpdfstring{Proof of \Cref{rigidityinlowlevel}}{Proof of Theorem \ref{rigidityinlowlevel}}}\label{sectionrigidityinlowlevel}

In this section, we prove \Cref{rigidityinlowlevel}. The proof is in $2$ steps. The first step is to prove rigidity for some small surfaces,
and is the content of \Cref{initializationeasy}, \Cref{S04} and \Cref{S05}.
The second step is to perform an induction on $g$ and $n$ as in the proof of \Cref{induction}.
However, because the assumptions of \Cref{mainlemma} are not always verified,
we have to adapt the techniques of the lemma to the various cases, see \Cref{S06}, \Cref{S0n} and proof of the theorem.

In the context of modular functors, we know that a boundary component colored with $0$ can be removed without changing the module.
However, when considering rigidity, the source group changes. The following lemma shows that this does not change
the space of deformations.

\begin{lemma}\label{removezeros}
    Let $\Nu$ be a modular functor of level $\ell>1$. 
    Let $S$ be a colored surface and $B\subset \partial S$ a boundary component colored with $0$.
    Let $\widehat{S}$ be the surface obtained by capping $S$ with a disk along $B$, ie. $\widehat{S}=S\cup_B D^2$.

    There is a natural isomorphism of $\tPModl{\ell}{S}$-modules
    $\Nu(S)\simeq\Nu(\widehat{S})$
    and a group morphism $\PModl{\ell}{S}\ra \PModl{\ell}{\widehat{S}}$,
    which induce an isomorphism:
    $$\coh{\PModl{\ell}{S}}{\ad\Nu(S)}\simeq\coh{\PModl{\ell}{\widehat{S}}}{\ad\Nu(\widehat{S})}.$$
\end{lemma}

\begin{proof}
    Let us remind the Birman exact sequence (\cite[4.2.5]{farbPrimerMappingClass2011}):
    $$1\lra\pi_1(U\widehat{S})\xlra{Push}\PMod{S}\lra\PMod{\widehat{S}}\lra 1.$$
    Here $U\widehat{S}$ denotes the unitary tangent bundle of $\widehat{S}$.
    Taking the quotient by the $\ell$-th powers of Dehn twists, we get an exact sequence:
    $$\pi_1(U\widehat{S})\xlra{Push}\PModl{\ell}{S}\lra\PModl{\ell}{\widehat{S}}\lra 1.$$
    Let us denote by $K$ the image of $Push$.

    The Inflation-Restriction sequence (\Cref{inflationrestriction}) is in this setting:
    $$0\ra\coh{\PModl{\ell}{\widehat{S}}}{\ad \Nu(\widehat{S})}\ra\coh{\PModl{\ell}{S}}{\ad \Nu(S)}\ra \Hom{K}{\ad \Nu(S)}.$$
    Now $\pi_1(U\widehat{S})$ is generated by $z$ and $\gamma_1,\dotsc,\gamma_d$ where $z$ is a simple loop on any fiber of
    $U\widehat{S}\lra \widehat{S}$ and the $\gamma_i$ are lifts of simple closed curves.
    As for their images in $\PModl{\ell}{S}$, $Push(z)$ is the Dehn twist along $B$ and $Push(\gamma_i)$ is a product
    of powers of two commuting Dehn twists. In particular, $Push(z)^\ell=1$ and $Push(\gamma_i)^\ell=1$ in $\PModl{\ell}{S}$.
    Thus $K^{\mathrm{ab}}$ is finite and $\Hom{K}{\ad \Nu(S)}=0$. Hence:
    $$\coh{\PModl{\ell}{\widehat{S}}}{\ad \Nu(\widehat{S})}\simeq\coh{\PModl{\ell}{S}}{\ad \Nu(S)}.$$
\end{proof}

From now on in this section, $\Nu$ will denote the modular functor associated to the $\SO$ TQFT in level $5$.
All boundary components will be colored with $2$.

\begin{lemma}\label{initializationeasy}
    The $\SO$ TQFT in level $5$ is locally rigid on $S_0^2$, $S_0^3$, $S_1$  and $S_1^1$. 
\end{lemma}

\begin{proof}
    The groups $\PModl{5}{S_0^2}$ and $\PModl{5}{S_0^3}$ are finite.
    The groups $\PMod{S_1}$ and $\PMod{S_{1,1}}$ are isomorphic to $\mathrm{SL}_2(\Z)$, via the linear action on $\R^2/\Z^2$.
    Moreover, one has $\mathrm{SL}_2(\Z)\simeq \Z/4*_{\Z/2}\Z/6$.
    This isomorphism is explicitly given by $S^{-1}$ and $ST$, where:
    $$S=\begin{pmatrix} 0 & -1 \\ 1 & 0 \end{pmatrix}\text{ and }T=\begin{pmatrix} 1 & 1 \\ 0 & 1 \end{pmatrix}.$$
    Then $T=S^{-1}ST$ is a Dehn twist in $\PMod{S_{1,1}}$ and $\PMod{S_1}$.
    Thus $T^5$ is trivial in $\PModl{5}{S_{1,1}}$ and $\PModl{5}{S_1}$.
    Hence the central quotients $\PModl{5}{S_{1,1}}/(\Z/2)$ and $\PModl{5}{S_1}/(\Z/2)$ are quotients of the triangular group $TG(2,3,5)$.
    As $1/2+1/3+1/5>1$, $TG(2,3,5)$ is spherical and thus finite.
    Hence $\PModl{5}{S_{1,1}}$ and $\PModl{5}{S_1}$ are finite.
    As $\PModl{5}{S_1^1}$ is an extension of $\PModl{5}{S_{1,1}}$ by $\Z/5$ or $\{1\}$, it is also finite.

    Finite groups do not have cohomology on $\C$-vector spaces,
    so the $\SO$ TQFT in level $5$ is necessarily locally rigid on $S_0^2$, $S_0^3$, $S_1$  and $S_1^1$.
\end{proof}

Here we used some explicit computations of Mapping Class Groups. They can be found in \cite[2.2]{farbPrimerMappingClass2011}.

\begin{lemma}\label{S04}
    The $\SO$ TQFT in level $5$ is locally rigid on $S_0^4$.
\end{lemma}

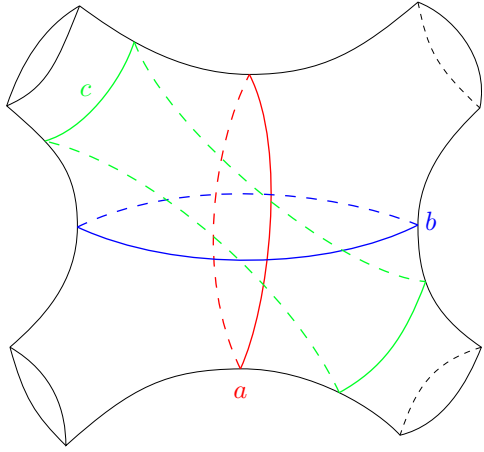
\begin{figure}
    \def\svgwidth{0.5\linewidth}
\begingroup%
  \makeatletter%
  \providecommand\color[2][]{%
    \errmessage{(Inkscape) Color is used for the text in Inkscape, but the package 'color.sty' is not loaded}%
    \renewcommand\color[2][]{}%
  }%
  \providecommand\transparent[1]{%
    \errmessage{(Inkscape) Transparency is used (non-zero) for the text in Inkscape, but the package 'transparent.sty' is not loaded}%
    \renewcommand\transparent[1]{}%
  }%
  \providecommand\rotatebox[2]{#2}%
  \newcommand*\fsize{\dimexpr\f@size pt\relax}%
  \newcommand*\lineheight[1]{\fontsize{\fsize}{#1\fsize}\selectfont}%
  \ifx\svgwidth\undefined%
    \setlength{\unitlength}{472.19346414bp}%
    \ifx\svgscale\undefined%
      \relax%
    \else%
      \setlength{\unitlength}{\unitlength * \real{\svgscale}}%
    \fi%
  \else%
    \setlength{\unitlength}{\svgwidth}%
  \fi%
  \global\let\svgwidth\undefined%
  \global\let\svgscale\undefined%
  \makeatother%
  \begin{picture}(1,0.93539099)%
    \lineheight{1}%
    \setlength\tabcolsep{0pt}%
    \put(0,0){\includegraphics[width=\unitlength,page=1]{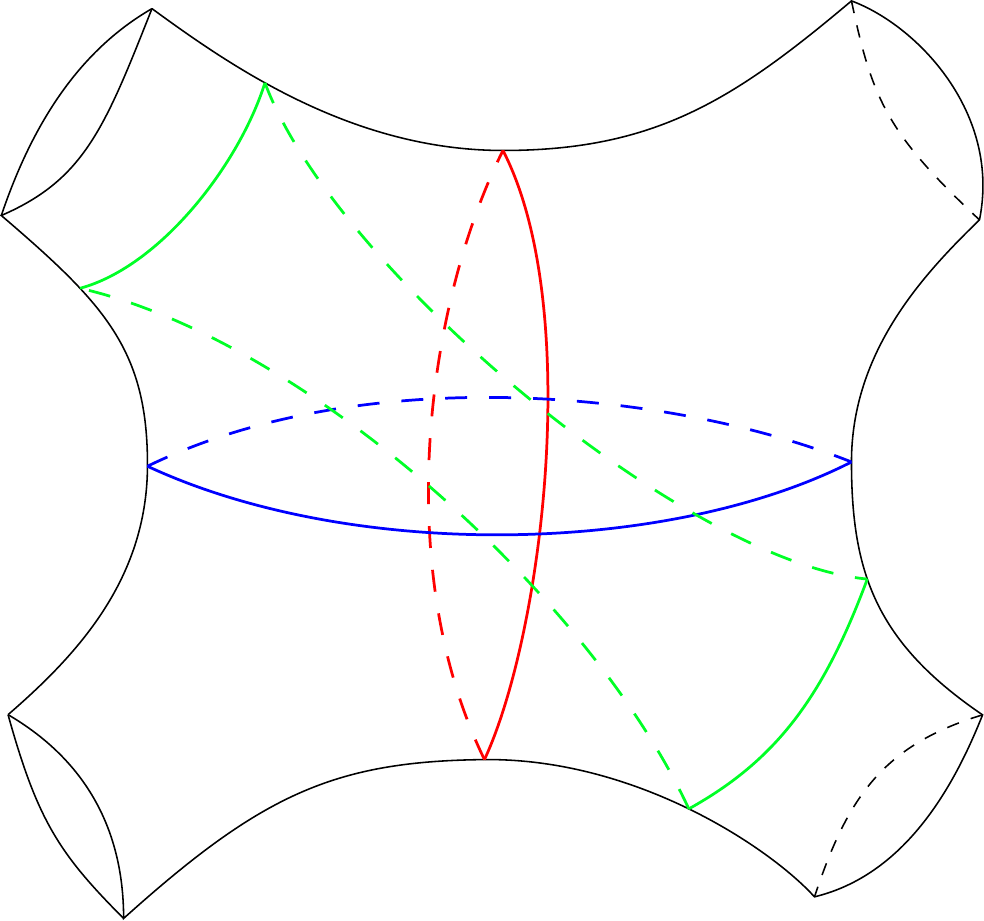}}%
    \put(0.47810842,0.10144443){\color[rgb]{1,0,0}\makebox(0,0)[lt]{\lineheight{1.25}\smash{\begin{tabular}[t]{l}$a$\end{tabular}}}}%
    \put(0.88105663,0.45571501){\color[rgb]{0,0,1}\makebox(0,0)[lt]{\lineheight{1.25}\smash{\begin{tabular}[t]{l}$b$\end{tabular}}}}%
    \put(0.15518932,0.73787733){\color[rgb]{0,1,0.14901961}\makebox(0,0)[lt]{\lineheight{1.25}\smash{\begin{tabular}[t]{l}$c$\end{tabular}}}}%
  \end{picture}%
\endgroup%

    \caption{Surface $S_0^4$ and curves.}
    \label[figure]{figure_S04}
\end{figure}

\begin{proof}
    Here $S=S_0^4$.
    Let $a$ and $b$ be the curves of \Cref{figure_S04}. Let $\Gamma_a$ be the stabilizer of $a$ and $\Gamma_b$ the stabilizer of $b$.
    Let $\Gamma_{ab}$ be their intersection. Define $G=\Gamma_a*_{\Gamma_{ab}}\Gamma_b$, and $R$ such that $\PModl{5}{S_0^4}=G/R$.
    Notice that $\Gamma_a$ and $\Gamma_b$ are finite groups. Hence, the Mayer-Vietoris sequence for $G$ is:
    $$0\lra M^G\lra M^{\Gamma_{a}}\oplus M^{\Gamma_{b}}\lra M^{\Gamma_{ab}}\lra\coh{G}{M}\lra 0$$
    where $M=\ad \Nu(S)$. Now, as $\Nu(S)$ has dimension $2$, a simple count on basis elements shows
    that $M^G$ has dimension $1$, $M^{\Gamma_{a}}$ and $M^{\Gamma_{b}}$ dimension $2$ and $M=M^{\Gamma_{ab}}$ dimension $4$.
    Hence $\coh{G}{M}$ has dimension $1$. Now the inflation restriction exact sequence for $R\subset G$ is:
    $$0\lra \coh{\PModl{5}{S_0^4}}{M}\lra\coh{G}{M}\lra\Hom{R}{M}.$$
    Let $\varphi\in\coh{G}{M}$ be non-zero. As $M^{\Gamma_{ab}}\lra\coh{G}{M}$ is surjective,
    a computation as in the proof of \Cref{mainlemma} gives that there exists $m\in M$ such that:
    $$\varphi(T_a)=T_a\cdot m-m\text{ and }\varphi(T_b)=0.$$

    Now, notice that $(T_aT_b)^{-1}$ is the Dehn twist $T_c$ on \Cref{figure_S04}. Hence $(T_aT_b)^5$ is in $R$.
    Let $\Phi(X)=1+X+X^2+X^3+X^4$. We have:
    $$\varphi((T_aT_b)^5)=\Phi(\ad (T_aT_b))(\varphi(T_aT_b))=\Phi(\ad (T_aT_b))(T_a\cdot m-m).$$
    As $(u-1)\Phi(u)=u^5-1=0$ for $u=\ad (T_aT_b)$, we have $\varphi((T_aT_b)^5)=0$
    if and only if $T_a\cdot m - m$ is in the image of $\ad (T_aT_b)- \id$.
    Let us now prove that:
    $$\dim(\mathrm{Im}(\ad T_a-\id)\cap \mathrm{Im}(\ad (T_aT_b)-\id))=1.$$
    Notice that for the usual trace quadratic form $(A,B)\mapsto \mathrm{Tr}(AB)$, we have:
    $$\mathrm{Im}(\ad T_a-\id)^\perp=\mathrm{Ker}(\ad T_a^{-1}-\id)=\mathrm{Ker}(\ad T_a-\id).$$
    Similarly for $\mathrm{Im}(\ad (T_aT_b)-\id))$.
    As scalar matrices are fixed by $\ad T_a$ and $\ad (T_aT_b)$, we have:
    $$\dim(\mathrm{Ker}(\ad T_a-\id)+\mathrm{Ker}(\ad (T_aT_b)-\id))\leq 3.$$
    As $T_a$ and $T_aT_b$ generate $\PModl{5}{S_0^4}$ and the representation is irreducible, the dimension is actually equal to $3$.
    Hence the dimension result. This shows that there exists $m\in M$ such that $T_a\cdot m-m\neq 0$
    and $\varphi((T_aT_b)^5)\neq 0$ for $\varphi$ associated to $m$ as above.

    From this we get that the map $M=M^{\Gamma_{ab}}\ra\coh{G}{M}\ra\Hom{R}{M}$ has rank at least $1$.
    Hence, as $\coh{G}{M}$ has dimension $1$, $\coh{G}{M}\ra\Hom{R}{M}$ is injective.
    From the inflation-restriction exact sequence above, we see that the space of deformations $\coh{\PModl{5}{S_0^4}}{M}$ is $0$, as desired.
\end{proof}

For the case of $S_{0,5}$, we use that the representation $\rho_{0,5}^5$ is the monodromy of the Hirzebruch surface.
This is proved in the paper of B. Deroin and J. Marché:

\begin{prop}{\cite[proposition 9]{deroinToledoInvariantsTopological2022}}\label{DeroinMarche}
    The $\SO$-quantum representation of level $5$ associated to $S_{0,5}$ and $\zeta_5=e^{\frac{2i\pi}{5}}$ is the
    holonomy of a $\mathbb{H}_{1,2}$-structure on the compact orbifold $\overline{\mathcal{M}}_{0,5}(5)$.
\end{prop}

We now retrieve local rigidity from Weil's rigidity.

\begin{lemma}\label{S05}
    The $\SO$ TQFT in level $5$ is locally rigid on $S_0^5$.
\end{lemma}

\begin{proof}
    A consequence of \Cref{DeroinMarche} is that the representation is given by the inclusion of a cocompact lattice $\Gamma\subset \PU{1}{2}$,
    and an isomorphism $\Gamma\simeq\PModl{5}{S_{0,5}}$.
    From the proof of local rigidity given by A. Weil \cite{weilRemarksCohomologyGroups1964},
    we have that $\coh{\Gamma}{\mathfrak{pu}(1,2)}=0$, where $\mathfrak{pu}(1,2)$ is the Lie algebra of $\PU{1}{2}$.
    Hence the representation is locally rigid in $\PU{1}{2}$.
    We conclude with \Cref{reductiontounitary}.
\end{proof}

\begin{lemma}\label{S06}
    The $\SO$ TQFT in level $5$ is locally rigid on $S_0^6$.
\end{lemma}

\begin{figure}
    \def\svgwidth{0.8\linewidth}
\begingroup%
  \makeatletter%
  \providecommand\color[2][]{%
    \errmessage{(Inkscape) Color is used for the text in Inkscape, but the package 'color.sty' is not loaded}%
    \renewcommand\color[2][]{}%
  }%
  \providecommand\transparent[1]{%
    \errmessage{(Inkscape) Transparency is used (non-zero) for the text in Inkscape, but the package 'transparent.sty' is not loaded}%
    \renewcommand\transparent[1]{}%
  }%
  \providecommand\rotatebox[2]{#2}%
  \newcommand*\fsize{\dimexpr\f@size pt\relax}%
  \newcommand*\lineheight[1]{\fontsize{\fsize}{#1\fsize}\selectfont}%
  \ifx\svgwidth\undefined%
    \setlength{\unitlength}{488.21230911bp}%
    \ifx\svgscale\undefined%
      \relax%
    \else%
      \setlength{\unitlength}{\unitlength * \real{\svgscale}}%
    \fi%
  \else%
    \setlength{\unitlength}{\svgwidth}%
  \fi%
  \global\let\svgwidth\undefined%
  \global\let\svgscale\undefined%
  \makeatother%
  \begin{picture}(1,0.53692585)%
    \lineheight{1}%
    \setlength\tabcolsep{0pt}%
    \put(0,0){\includegraphics[width=\unitlength,page=1]{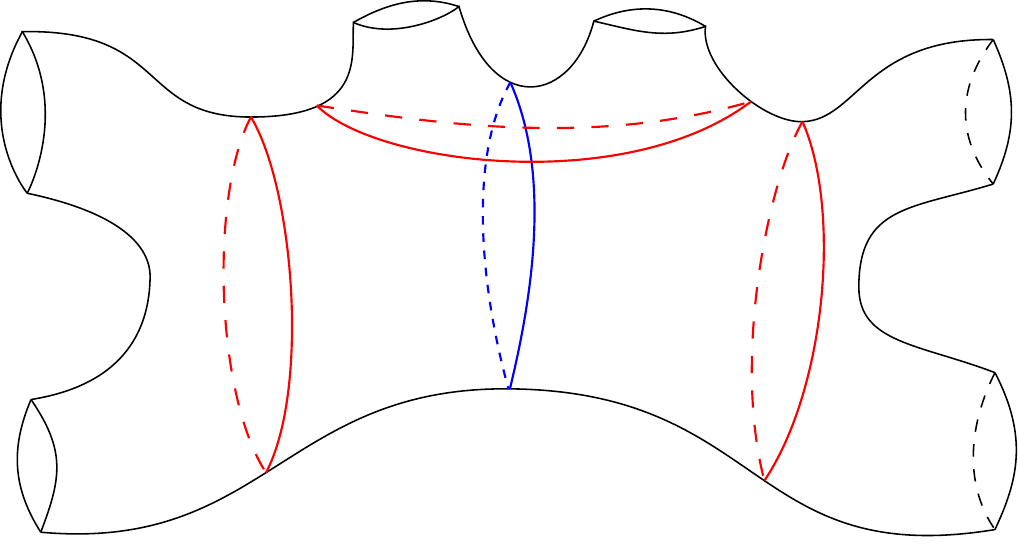}}%
    \put(0.24672689,0.01755443){\color[rgb]{1,0,0}\makebox(0,0)[lt]{\lineheight{1.25}\smash{\begin{tabular}[t]{l}$a_1$\end{tabular}}}}%
    \put(0.7388621,0.00722104){\color[rgb]{1,0,0}\makebox(0,0)[lt]{\lineheight{1.25}\smash{\begin{tabular}[t]{l}$a_2$\end{tabular}}}}%
    \put(0.62648482,0.3469364){\color[rgb]{1,0,0}\makebox(0,0)[lt]{\lineheight{1.25}\smash{\begin{tabular}[t]{l}$a_3$\end{tabular}}}}%
    \put(0.49085697,0.11055651){\color[rgb]{0,0,1}\makebox(0,0)[lt]{\lineheight{1.25}\smash{\begin{tabular}[t]{l}$c$\end{tabular}}}}%
  \end{picture}%
\endgroup%

    \caption{Surface $S_0^6$ and curves.}
    \label[figure]{figure_S06}
\end{figure}
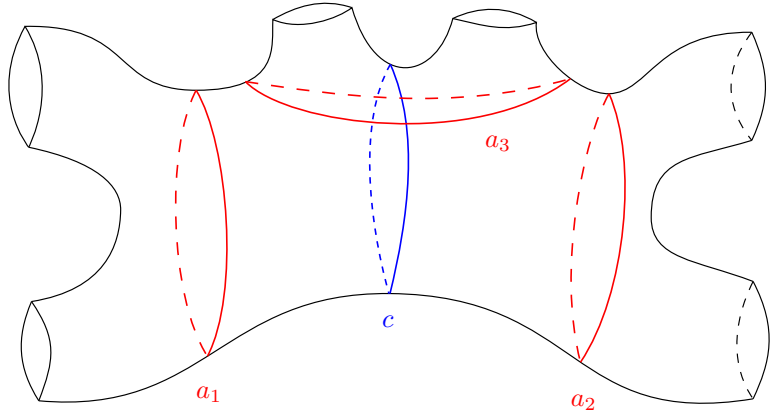

\begin{proof}
    Here $S=S_0^6$.
    We will need a slight variation of \Cref{mainlemma} where the mapping class group is not generated
    by the stabilizer of $a_1$ and the stabilizer of $a_2$.
    The proof is exactly the same if we replace the full mapping class group by the subgroup generated by the stabilizers.

    Let $a_1$, $a_2$, $a_3$, and $c$ be the {\sccs} on \Cref{figure_S06}.

    Let $\Gamma_i$ be the stabilizer of $a_i$ for $i=1,2,3$. Let $\Gamma_{ij}$ be the stabilizer of $a_i$ and $a_j$ for $i,j\in\{1,2,3\}$.
    Let $G_{ij}=\Gamma_i*_{\Gamma_{ij}}\Gamma_j$ for $i,j\in\{1,2,3\}$. Let $L_{ij}$ be the image of $G_{ij}\ra \PModl{\ell}{S_0^6}$.

    The curves $a_1$, $a_2$ and $c$ verify all the hypotheses of \Cref{mainlemma} for $\lambda_{c}=2$
    if we replace $\PModl{\ell}{S_0^6}$ by $L_{12}$. Hence, applying the proof of the lemma, we get:
    $$\coh{L_{12}}{\ad\Nu(S)}=0.$$
    By symmetry, the same result holds for $L_{13}$.

    Let $G=L_{12}*_{\langle \Gamma_1,\Gamma_{23}\rangle}L_{13}$. Now $\PModl{\ell}{S_0^6}$ is a quotient of $G$.
    Moreover, the Mayer-Vietoris sequence is:
    $$0\lra M^G\lra M^{L_{12}}\oplus M^{L_{13}}\lra M^{\langle \Gamma_1,\Gamma_{23}\rangle}\lra\coh{G}{M}\lra0$$
    where $M=\ad \Nu(S)$. Clearly, by \Cref{irreducibility}, $M^G$ has dimension $1$.

    Now, as all $4$ colorings $(0,0)$, $(0,2)$, $(2,0)$ and $(2,2)$ of $(a_1,a_2)$ are admissible,
    $\Nu(S)$ is irreducible as a representation of $G_{12}=\Gamma_1*_{\Gamma_{12}}\Gamma_2$. The same holds for $G_{13}$. 
    Thus, $M^{L_{12}}$ and $M^{L_{13}}$ have dimension $1$.

    To conclude, we must prove that $M^{\langle \Gamma_1,\Gamma_{23}\rangle}$ also has dimension $1$.
    Let $\varphi$ be an element of $M$ stabilized by $\Gamma_1$ and $\Gamma_{23}$.
    Then, $\varphi$ is diagonal in the basis associated to the pair of pants decomposition formed by the $a_i$.
    Let $\lambda(c_1,c_2,c_3)$ be its diagonal coefficient on the basis element where $a_i$ is colored with $c_i$.

    As $\Gamma_{23}$ fixes $\varphi$, by \Cref{irreducibility}, $\lambda(c_1,c_2,c_3)$ only depends on $(c_2,c_3)$.
    But the colorings $(2,2,2)$ and $(0,2,2)$ are admissible, so $\lambda(0,2,2)=\lambda(2,2,2)$.
    As $\Gamma_1$ fixes $\varphi$, by \Cref{irreducibility}, $\lambda(c_1,c_2,c_3)$ only depends on $c_1$.
    So $\varphi$ is scalar. Hence the result.
\end{proof}

\begin{lemma}\label{S0n}
    For $n\geq 7$, the $\SO$ TQFT in level $5$ is locally rigid on $S_0^n$.
\end{lemma}

\begin{figure}
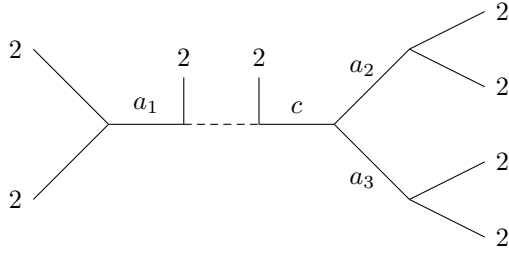

    \ctikzfig{graph_for_0n}
    \caption{Graph of a pair of pants decomposition of $S_0^n$ for $n\geq 7$.}
    \label[figure]{graph_for_0n}
\end{figure}

\begin{proof}
    Here $S=S_0^n$.
    The proof is very similar to the one given for $S_0^6$.
    We consider the pair of pants decomposition associated to the graph of \Cref{graph_for_0n}.
    Let $a_1$, $a_2$, $a_3$, and $c$ be the {\sccs} associated to the edges marked on the figure.

    Let $\Gamma_i$, $\Gamma_{ij}$, $L_{ij}$ and $G$ be as in the proof of \Cref{S06}.

    As for $S_0^6$, the curves $a_1$, $a_2$ and $c$ verify all the hypotheses of \Cref{mainlemma} for $\lambda_{c}=2$
    if we replace $\PModl{\ell}{S_0^n}$ by $L_{12}$. Hence, we get:
    $$\coh{L_{12}}{\ad\Nu(S)}=\coh{L_{13}}{\ad\Nu(S)}=0.$$

    The Mayer-Vietoris sequence is:
    $$0\lra M^G\lra M^{L_{12}}\oplus M^{L_{13}}\lra M^{\langle \Gamma_1,\Gamma_{23}\rangle}\lra\coh{G}{M}\lra0$$
    where $M=\ad \Nu(S)$. As for $S_0^6$, $M^G$, $M^{L_{12}}$ and $M^{L_{13}}$ have dimension $1$.

    Now as all $8$ possible colorings of $(a_1,a_2,a_3)$ appear in the basis,
    $\Nu(S)$ is irreducible as a representation of $\langle \Gamma_1,\Gamma_{23}\rangle$.
    Thus, $M^{\langle \Gamma_1,\Gamma_{23}\rangle}$ has dimension $1$.

    The result follows by dimension count on the exact sequence.
\end{proof}

We now have all the base cases necessary to prove the theorem.

\begin{proof}[Proof of \Cref{rigidityinlowlevel}.]
    The second statement of the theorem follows from the first and \Cref{removingthel}.
    
    In this proof, for any of the figures we will use the following notations.
    For $i\in\{1,2\}$, $\Gamma_i=\PModl{\ell}{S,[a_i]}$ will denote the stabilizer of $a_i$.
    $\Gamma_{12}=\PModl{\ell}{S,[a_1],[a_2]}$ is their intersection.
    $G=\Gamma_1*_{\Gamma_{12}}\Gamma_2$ will be their amalgamated sum.

    \textbf{(1) $\mathbf{S_1^2}$.}
    \begin{figure}
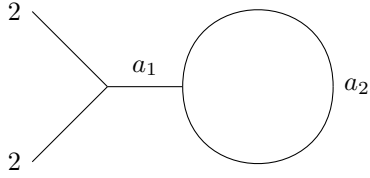

        \ctikzfig{graph_for_12}
        \caption{Graph of a pair of pants decomposition of $S_1^2$.}
        \label[figure]{graph_for_12}
    \end{figure}
    Here $S = S_1^2$. We use \Cref{graph_for_12} as a reference.
    By \Cref{initializationeasy} and \Cref{S04},
    $\Nu$ is locally rigid on the components of $S_{a_1}\simeq S_0^3\sqcup S_1^1$ and on $S_{a_2}\simeq S_0^4$
    for any coloring of the boundary components. Hence, by \Cref{crossterms}:
    $$\coh{\Gamma_1}{\ad\Nu(S)}=0\text{ and }\coh{\Gamma_2}{\ad\Nu(S)}=0.$$
    Now, according to \Cref{mayervietoris}, we have an exact sequence:
    $$0\lra M^G\lra M^{\Gamma_1}\oplus M^{\Gamma_2}\lra M^{\Gamma_{12}}\lra\coh{G}{M}\lra0.$$
    Here $M=\ad\Nu(S)$. Now, a quick count on the basis using \Cref{irreducibility}
    shows that $M^{\Gamma_1}$ and $M^{\Gamma_2}$ have dimension $2$,
    whereas $M^{\Gamma_{12}}$ has dimension $3$ and $M^G$ has dimension $1$. Hence $\coh{G}{M}=0$, and by \Cref{inflationrestriction},
    $\coh{\Gamma}{M}=0$, as required.

    \textbf{(2) $\mathbf{S_1^3}$.}
    \begin{figure}
        \ctikzfig{graph_for_13}
        \caption{Graph of a pair of pants decomposition of $S_1^3$.}
        \label[figure]{graph_for_13}
    \end{figure}
    Here $S = S_1^3$. We use \Cref{graph_for_13} as a reference.
    By \textbf{(1)}, \Cref{initializationeasy} and \Cref{S05},
    $\Nu$ is locally rigid on the components of $S_{a_1}\simeq S_0^3\sqcup S_1^2$ and on $S_{a_2}\simeq S_0^5$
    for any coloring of the boundary components.

    As above, we get the exact sequence:
    $$0\lra M^G\lra M^{\Gamma_1}\oplus M^{\Gamma_2}\lra M^{\Gamma_{12}}\lra\coh{G}{M}\lra0.$$
    And again $M^{\Gamma_1}$ and $M^{\Gamma_2}$ have dimension $2$, $M^{\Gamma_{12}}$ has dimension $3$ and $M^G$ has dimension $1$.
    So $\coh{\Gamma}{M}=0$.

    \textbf{(3) $\mathbf{S_1^n}$ with $\mathbf{n\geq 4}$.}
    Fix $n\geq 4$.
    \begin{figure}
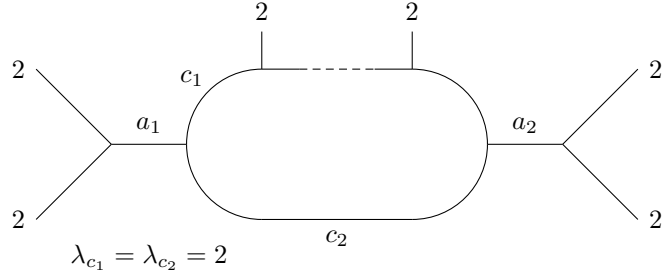

        \ctikzfig{graph_for_1n}
        \caption{Graph of a pair of pants decomposition of $S_1^n$.}
        \label[figure]{graph_for_1n}
    \end{figure}
    Here $S = S_1^n$. We use \Cref{graph_for_1n} as a reference.
    We apply \Cref{mainlemma} with $a_1$, $a_2$ and $c$ as on the figure.
    One checks that the hypotheses of the lemma are verified when the components of $c$ are colored with $2$.
    The only non-trivial part is the irreducibility. But as we are in level $5$, all the representations are irreducible,
    see \Cref{irreducibility}.
    
    As $S_{a_1}\simeq S_{a_2}\simeq S_0^3\sqcup S_1^{n-1}$, we are reduced to the case of $S_1^{n-1}$.
    We conclude by \textbf{(2)} and induction on $n$.

    \textbf{(4) $\mathbf{S_2^1}$.}
    \begin{figure}
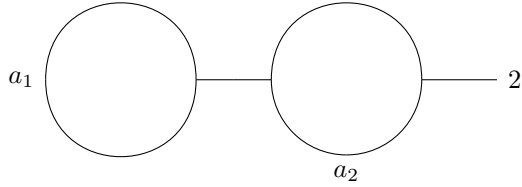

        \ctikzfig{graph_for_21}
        \caption{Graph of a pair of pants decomposition of $S_2^1$.}
        \label[figure]{graph_for_21}
    \end{figure}
    Here $S = S_2^1$. We use \Cref{graph_for_21} as a reference.
    We proceed as in \textbf{(1)} and \textbf{(2)}. Using \Cref{initializationeasy} and \textbf{(2)}, we see that $\Nu$
    is locally rigid on $S_{a_1}\simeq S_{a_2}\simeq S_1^3$ for any coloring of the $a_i$.
    Hence, we get the exact sequence:
    $$0\lra M^G\lra M^{\Gamma_1}\oplus M^{\Gamma_2}\lra M^{\Gamma_{12}}\lra\coh{G}{M}\lra0.$$
    And again $M^{\Gamma_1}$ and $M^{\Gamma_2}$ have dimension $2$, $M^{\Gamma_{12}}$ has dimension $3$ and $M^G$ dimension $1$.
    So $\coh{\Gamma}{M}=0$.

    \textbf{(5) $\mathbf{S_2^n}$ with $\mathbf{n=0}$ or $\mathbf{n\geq 2}$.}
    Fix $n\geq 2$. The cas $n=0$ is treated similarly.
    \begin{figure}
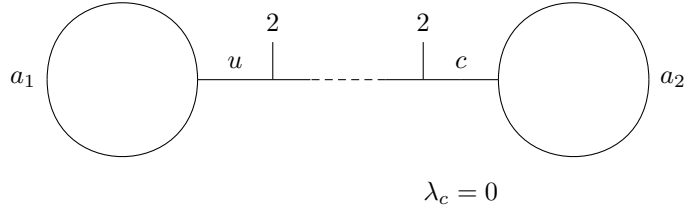

        \ctikzfig{graph_for_2n}
        \caption{Graph of a pair of pants decomposition of $S_2^n$, $n\geq 2$.}
        \label[figure]{graph_for_2n}
    \end{figure}
    Here $S = S_2^n$. We use \Cref{graph_for_2n} as a reference in $\textbf{(5)}$.
    We easily check that $a_1$, $a_2$ and $c$ as on the figure verify the hypotheses of \Cref{mainlemma} for $\lambda_c=0$.
    As in \textbf{(3)}, we use \Cref{irreducibility}.
    
    As $S_{a_1}\simeq S_{a_1}\simeq S_1^{n+2}$, \textbf{(3)} and \Cref{bettermainlemma} imply that $\Nu$ is locally rigid on $S$.

    \textbf{(6) $\mathbf{S_g^n}$ with $\mathbf{g\geq 3}$.}
    Choose $g\geq 3$ and $n\geq 0$. We proceed by induction on $g$.
    Here $S=S_g^n$ and we use \Cref{graph_for_induction_on_g} as a reference.
    The hypotheses of \Cref{mainlemma} can be checked exactly as in the proof of \Cref{induction}.
    As $S_{a_1}\simeq S_{a_1}\simeq S_{g-1}^{n+2}$, the induction hypothesis,
    \textbf{(4)}, \textbf{(5)} and \Cref{bettermainlemma} imply that $\Nu$ is locally rigid on $S$.
\end{proof}

\appendix

\section{Irreducibility of quantum representations}\label{appendixirreducibility}

In this appendix, we prove irreducibility of quantum $\SO$ representations at prime levels.
We deduce it as a corollary of a result from
the paper \cite{koberdaIrreducibilityQuantumRepresentations2018} of T. Koberda and R. Santharoubane.
In this section, we denote $\Nu_\ell$ the TQFT of level $\ell$ constructed in \cite{blanchetTopologicalQuantumField1995}.

Here is the result we need from the paper of T. Koberda and R. Santharoubane, stated in the case of a $\SO$ TQFT.
Note that this is not the main result of their paper.

\begin{prop}{\cite[3.2]{koberdaIrreducibilityQuantumRepresentations2018}}\label{propramanujan}
    Let $\ell\geq 5$ be an odd integer.
    Let $S_g^n$ denote the compact surface of genus $g$ with $n$ boundary components.
    Denote by $\mathcal{S}(S_g^n)$ the Kauffman bracket skein module of $S_g^n\times [0,1]$.

    Let $\lambda_1,\dotsc,\lambda_n\in\{0,2,\dotsc,\ell-3\}$.   
    Denote by $(S_{g,n},\lambda_1,\dotsc,\lambda_n)$ the closed surface of genus $g$ with $n$ marked points labelled $\lambda_1,\dotsc,\lambda_n$.

    Then the map:
    $$\mathcal{S}(S_g^n)\lra \End{\Nu_\ell(S_{g,n},\lambda_1,\dotsc,\lambda_n)}$$
    is surjective.
\end{prop}

\begin{remark}
    Let $\beta\subset S_g^n$ be a simple closed curve. Let $C_\beta$ be the curve operator
    that acts on $\Nu_\ell(S_{g,n},\lambda_1,\dotsc,\lambda_n)$ via the element $e_\beta$ of $\mathcal{S}(S_g^n)$ corresponding to $\beta$.
    If $\ell$ is prime, the operator $C_\beta$ is a polynomial in the Dehn twist $T_\beta$
    (see \cite[4.2]{marcheIntroductionQuantumRepresentations2021}).
\end{remark}

We can now apply the proof of irreducibility of Roberts \cite{robertsIrreducibilityQuantumRepresentations2001}
to the case of surfaces with boundary.

\begin{coro}\label{irreducibility}
    Let $\ell\geq 5$ be a prime number. Let $g,n\geq 0$ and $\underline{\lambda}\in\Lambda^n$.
    Then the associated quantum $\SO$ representation of $\PMod{S_g^n}$ is irreducible.
\end{coro}

\begin{proof}
    The set $\{e_\beta\mid\beta\text{ \scc}\}$ generates the skein module $\mathcal{S}(S_g^n)$\linebreak as an algebra.
    Hence, by \Cref{propramanujan}, the endomorphism algebra\linebreak $\End{\Nu_\ell(S_{g,n},\lambda_1,\dotsc,\lambda_n)}$
    is generated by the curve operators.
    Moreover, every such operator $C_\beta$ is a polynomial in the associated Dehn twist $T_\beta$,
    so that\linebreak$\End{\Nu_\ell(S_{g,n},\lambda_1,\dotsc,\lambda_n)}$ is generated as an
    algebra by the action of the Dehn twists.
    Thus, the map:
    $$\C[\PMod{S_g^n}]\lra \End{\Nu_\ell(S_{g,n},\lambda_1,\dotsc,\lambda_n)}$$
    is surjective.
    This shows that $\Nu_\ell(S_{g,n},\lambda_1,\dotsc,\lambda_n)$ is irreducible.
\end{proof}

\bibliographystyle{plain}
\bibliography{biblio}

\end{document}